\documentclass[10pt,reqno]{amsart}
\usepackage{amscd,amsmath,amssymb,amsthm,bbm}
\title[Real extensions of distal minimal flows]{Real extensions of distal minimal flows and continuous topological ergodic decompositions}
\author{Gernot Greschonig}
\address{Institute of Discrete Mathematics and Geometry, Vienna University of Technology, Wiedner Hauptstra\ss e 8-10, A-1040 Wien, Austria}
\email{greschg@fastmail.net}
\newtheorem{theorem}{Theorem}[section]
\newtheorem*{strth*}{Structure theorem}
\newtheorem*{decth*}{Decomposition theorem}
\newtheorem*{tmath*}{Topological Mackey action}

\newtheorem*{cor*}{Corollary}
\newtheorem{lemma}[theorem]{Lemma}
\newtheorem{fact}[theorem]{Fact}

\newtheorem{proposition}[theorem]{Proposition}
\theoremstyle{definition}
\newtheorem{defi}[theorem]{Definition}
\theoremstyle{remark}
\newtheorem{remark}[theorem]{Remark}
\newtheorem{remarks}[theorem]{Remarks}
\newtheorem{remarks*}[]{Remark}
\newtheorem{example}[theorem]{Example}
\newcommand{\comment}[1]{}

\newcommand{\N}{\mathbb N}
\newcommand{\Z}{\mathbb Z}
\newcommand{\T}{\mathbb T}
\newcommand{\R}{\mathbb R}

\newcommand{\Q}{\mathbb Q}
\newcommand{\mc}{\mathcal}


\begin{document}\allowdisplaybreaks\frenchspacing

\thanks{The author was supported by the FWF research grants J2622 and S9600, and by THE ISRAELI SCIENCE FOUNDATION grant No. 114/08}

\begin{abstract}
We prove a structure theorem for topologically recurrent real skew product extensions of distal minimal compact metric flows with a compactly generated Abelian acting group (e.g. $\Z^d$-flows and $\R^d$-flows).
The main result states that every such extension apart from a coboundary can be represented by a perturbation of a so-called Rokhlin skew product.
We obtain as a corollary that the topological ergodic decomposition of the skew product extension into prolongations is continuous and compact with respect to the Fell topology on the hyperspace.
The right translation acts minimally on this decomposition, therefore providing a minimal compact metric analogue to the Mackey action.
This topological Mackey action is a distal (possibly trivial) extension of a weakly mixing factor (possibly trivial), and it is distal if and only if perturbation of the Rokhlin skew product is defined by a topological coboundary.
\end{abstract}

\maketitle

\section{Introduction and main results}
The study of real-valued topological cocycles and real skew product extensions has been initiated by Besicovitch, Gottschalk, and Hedlund.
Besicovitch \cite{Be} proved the existence of point transitive real skew product extensions of an irrational rotation on the one-dimensional torus.
Furthermore, he proved that none of them is minimal, i.e. there are always non-transitive points for a point transitive real skew product extension.
The main result in Chapter 14 of \cite{G-H} can be rephrased to the dichotomy that a topologically conservative real skew product extension of a minimal rotation on a torus (finite or infinite dimensional) is either point transitive or it is defined by a topological coboundary and almost periodic.
This result and a generalisation to skew product extensions of a Kronecker transformation (cf. \cite{LM}) exploit the isometric behaviour of a minimal rotation.
A corresponding result apart from isometries is based on homotopy conditions for the class of distal minimal homeomorphisms usually called Furstenberg transformations (cf. \cite{Gr}).
However, in general this dichotomy is not valid, and counterexamples can be provided by the Rokhlin skew products of the so-called topological type $III_0$.
This motivates the study of topologically conservative real skew product extensions of compact flows apart from isometries and toral extensions, which is carried out in this note for \emph{distal} minimal flows with Abelian compactly generated acting groups.

Throughout this note we shall denote by $T$ a \emph{compactly generated Abelian} Hausdorff topological group acting continuously on a compact metric phase space $(X,d)$ so that $(X,T)$ is a \emph{compact metric flow}.
In the monograph \cite{G-H} such an acting group $T$ is called \emph{generative}, and notions of recurrence are provided for such Abelian acting groups apart from $\Z$ and $\R$.
For a $\Z$-action on $X$ we let $T$ be the self-homeomorphism of $X$ generating the action by $(n,x)\mapsto T^n x$, while in the case of a real flow we shall use the notation $\{\phi^t:t\in\R\}$ for the acting group.
We call a flow \emph{minimal} if the whole phase space is the only non-empty invariant closed subset, and then for every $x\in X$ the \emph{orbit closure} $\bar{\mc O}_T(x)=\overline{\{\tau x:\tau\in T\}}$ is all of $X$.
A flow $(X,T)$ is \emph{topologically transitive} if for arbitrary open neighbourhoods $\mc U,\mc V\subset X$ there exists some $\tau\in T$ with $\tau\mc U\cap\mc V\neq\emptyset$, and it is \emph{weakly mixing} if the flow $(X\times X,T)$ with the diagonal action is topologically transitive.
For a topologically transitive flow $(X,T)$ with complete separable metric phase space there exists a dense $G_\delta$-set of \emph{transitive points} $x$ with $\bar{\mc O}_T(x)=X$, and a flow with transitive points is \emph{point transitive}.
If $(X,T)$ and $(Y,T)$ are flows with the same acting group $T$ and $\pi:X\longrightarrow Y$ is a continuous \emph{onto} mapping with $\pi(\tau x)=\tau\pi(x)$ for every $\tau\in T$ and $x\in X$, then $(Y,T)=\pi(X,T)$ is called a factor of $(X,T)$ and $(X,T)$ is called an extension of $(Y,T)$.
Such a mapping $\pi$ is called a \emph{homomorphism} of the flows $(X,T)$ and $(Y,T)$.
The set of bicontinuous bijective homomorphisms of a flow $(X,T)$ onto itself is the topological group $\textup{Aut}(X,T)$ of \emph{automorphisms} of $(X,T)$ with the topology of uniform convergence.
Two points $x,y\in X$ are called \emph{distal} if
\begin{equation*}
\inf_{\tau\in T}d(\tau x,\tau y)> 0 ,
\end{equation*}
otherwise they are called \emph{proximal}.
For a general compact Hausdorff flow $(X,T)$ distality of two points $x,y\in X$ is defined by the absence of any nets $\{\tau_n\}_{n\in I}\subset T$ with $\lim \tau_n x=\lim \tau_n y$.
A flow is called distal if any two distinct points are distal, and an extension of flows is called distal if any two distinct points in the same fibre are distal.
An important property of distal compact flows is the \emph{partitioning} of the phase space into invariant closed minimal subsets, even if the flow is not minimal.

Suppose that $\mathbb A$ is an Abelian locally compact second countable (Abelian l.c.s.) group with zero element $\mathbf 0_\mathbb A$, and let $\mathbb A_\infty$ denote its one point compactification with the convention that $g+\infty=\infty+g=\infty$ for every $g\in \mathbb A$.
A cocycle of a compact metric flow $(X,T)$ is a continuous mapping $f:T\times X\longrightarrow \mathbb A$ with the identity
\begin{equation*}
f(\tau,\tau' x)+f(\tau',x)=f(\tau\tau',x)
\end{equation*}
for all $\tau,\tau'\in T$ and $x\in X$.
Given a compact metric $\Z$-flow $(X,T)$ and a continuous function $f:X\longrightarrow \mathbb A$, we can define a cocycle $f:\mathbb{Z}\times X \longrightarrow \mathbb A$ with $f(1,\cdot)\equiv f$ by
\begin{equation*}
f(n,x)=
\begin{cases}
\sum_{k=0}^{n-1}f(T^k x) & \textup{if}\enspace n\geq 1 ,
\\
\mathbf 0_\mathbb A & \textup{if}\enspace n=0 ,
\\
-f(-n,T^n x) & \textup{if}\enspace n < 0.
\end{cases}
\end{equation*}
Moreover, there is a natural occurrence of cocycles of $\R$-flows as solutions to ODE's.
Suppose that $(M,\{\phi^t:t\in\R\})$ is a smooth flow on a compact manifold $M$ and $A:M\longrightarrow\R$ is a smooth function.
Then a continuous real valued cocycle $f(t,m)$ of the flow $(M,\{\phi^t:t\in\R\})$ is given by the fundamental solution to the ODE
\begin{equation*}
\frac{d\, f(t,m)}{d\, t}=A(\phi^t(m))
\end{equation*}
with the initial condition $f(0,m)=0$.
The \emph{skew product extension} of the flow $(X,T)$ by a cocycle $f:T\times X\longrightarrow \mathbb A$ is defined by the homeomorphisms
\begin{equation*}
\widetilde\tau_f(x,a)=(\tau x, f(\tau,x)+ a)
\end{equation*}
of $X\times \mathbb A$ for all $\tau\in T$, which provide a continuous action $(\tau,x,a)\mapsto\widetilde\tau_f(x,a)$ of $T$ on $X\times \mathbb A$ by the cocycle identity.
For a $\Z$-flow $(X,T)$ this action is generated by
\begin{equation*}
\widetilde T_f(x,a)=(T x, f(x)+ a) .
\end{equation*}
The essential property of a skew product is that the $T$-action on $X\times\mathbb A$ commutes with the right translation action of the group $\mathbb A$ on $X\times\mathbb A$, which is defined by $$R_b(x,a)=(x,a-b)$$
 for every $b\in\mathbb A$.
The orbit closure of $(x,a)\in X\times\mathbb A$ under $\widetilde\tau_f$ will be denoted by $\bar{\mc O}_{T,f}(x,a)=\overline{\{\widetilde\tau_f(x,a):\tau\in T\}}$.

The \emph{prolongation} $\mc D_T(x)$ of $x\in X$ under the group action of $T$ is defined by
\begin{equation*}
\mc D_T(x)=\bigcap\{\bar{\mc O}_T(\mc U):\mc U\enspace\textup{is an open neighbourhood of}\enspace x\},
\end{equation*}
and we shall use the notation $\mc D_{T,f}(x,a)$ for the prolongation of a point $(x,a)\in X\times \mathbb A$ under the skew product action $\widetilde\tau_f$.

While the inclusion of the orbit closure in the prolongation is obvious, the coincidence of these sets is generic by a result from the paper \cite{Gl3}.
This result, one of our main tools, is usually referred to as ``topological ergodic decomposition''.

\begin{fact}\label{fact:o_p}
For every compact \emph{metric} flow $(X,T)$ there exists a $T$-invariant dense $G_\delta$ set $\mc F\subset X$ so that for every $x\in\mc F$ holds
\begin{equation*}
\bar{\mc O}_{T}(x)=\mc D_{T}(x) .
\end{equation*}
For a skew product extension $\widetilde\tau_f$ of $(X,T)$ by a cocycle $f:T\times X\longrightarrow \mathbb A$ there exists a $T$-invariant dense $G_\delta$ set $\mc F\subset X$ so that for every $x\in\mc F$ and \emph{every} $a\in \mathbb A$ holds
\begin{equation*}
\bar{\mc O}_{T,f}(x,a)=\mc D_{T,f}(x,a) .
\end{equation*}
This assertion holds as well for the extension of $\widetilde\tau_f$ to $X\times \mathbb A_\infty$ which is defined by $(x,\infty)\mapsto (\tau x,\infty)$ for every $x\in X$, and given an $\R^2$-valued topological cocycle $g=(g_1,g_2):T\times X\longrightarrow\R^2$ for the extension of $\widetilde\tau_g$ to $X\times(\R_\infty)^2$ which is defined by $(x,s,\infty)\mapsto (\tau x,s+g_1(x),\infty)$, $(x,\infty,t)\mapsto (\tau x,\infty,t+g_2(x))$, and $(x,\infty,\infty)\mapsto (\tau x,\infty,\infty)$, for every $x\in X$ and $s,t\in\R$.
\end{fact}

\begin{proof}
The statement for a compact metric phase space and a general acting group is according to Theorem 1 of \cite{AkGl}.
The other statements can be verified by means of the extension of $\widetilde\tau_f$ onto the compactification of $X\times\mathbb A$.
The coincidence of $\bar{\mc O}_{T,f}(x,a)$ and $\mc D_{T,f}(x,a)$ for some $(x,a)\in X\times \mathbb A$ implies this coincidence for all $(x,a')\in\{x\}\times \mathbb A_\infty$, since the extension of $\widetilde\tau_f$ to $X\times\mathbb A_\infty$ commutes with the right translation on $X\times \mathbb A_\infty$.
\end{proof}

\begin{remark}\label{rem:o_p}
If $y\in\bar{\mc O}_{T}(x)$ and $z\in\bar{\mc O}_{T}(y)$, then $z\in\bar{\mc O}_{T}(x)$ follows by a diagonalisation argument.
A corresponding statement for prolongations is not valid, however follows from $x\in\bar{\mc O}_{T}(y)$ and $z\in\mc D_{T}(y)$ that $z\in\mc D_{T}(x)$.
\end{remark}

We shall consider more general Abelian acting groups than $\Z$ and $\R$, hence the definition of recurrence requires the notions of a replete semigroup and an extensive subset of the Abelian compactly generated group $T$ (cf. \cite{G-H}).
We recall that a semigroup $P\subset T$ is replete if for every compact subset $K\subset T$ there exists a $\tau\in T$ with $\tau K\subset P$, and a subset $E\subset T$ is extensive if it intersects every replete semigroup.
Therefore, a subset $E$ of $T=\Z$ or $T=\R$ is extensive if and only if $E$ contains arbitrarily large positive \emph{and} arbitrarily large negative elements.

\begin{defi}
We call a cocycle $f(\tau,x)$ of a minimal compact metric flow $(X,T)$ \emph{topologically recurrent} if for arbitrary neighbourhoods $\mc U\subset X$ and $U(\mathbf 0_\mathbb A)\subset \mathbb A$ of $\mathbf 0_\mathbb A$ there exists an extensive set of elements $\tau\in T$ with
\begin{equation*}
\mc U\cap\tau^{-1}(\mc U)\cap\{x\in X: f(\tau,x)\in U(\mathbf 0_\mathbb A)\}\neq\emptyset .
\end{equation*}
Since $\widetilde\tau_f$ and the right translation on $X\times\mathbb A$ commute, this is equivalent to the \emph{regional recurrence} of the skew product action $\widetilde\tau_f$ on $X\times\mathbb A$, i.e. for every open neighbourhood $U\subset X\times \mathbb A$ there exists an extensive set of elements $\tau\in T$ with $\widetilde\tau_f(U)\cap U\neq\emptyset$.
A non-recurrent cocycle is called \emph{transient}.

A point $(x,a)\in X\times \mathbb A$ is $\widetilde\tau_f$-recurrent if for every neighbourhood $U\subset X\times \mathbb A$ of $(x,a)$ the set of $\tau\in T$ with $\widetilde\tau_f(x,a)\in U$ is extensive.
Moreover, a point $(x,a)\in X\times\mathbb A$ is \emph{regionally} $\widetilde\tau_f$-recurrent if for every neighbourhood $U$ of $(x,a)$ the set of $\tau\in T$ with $\widetilde\tau_f(U)\cap U\neq\emptyset$ is extensive.
\end{defi}

\begin{remarks}\label{rems:rec}
If $f(\tau,x)$ is recurrent, then by Theorems 7.15 and 7.16 in \cite{G-H} there exists a dense $G_\delta$ set of $\widetilde\tau_f$-recurrent points in $X\times\R$.

Given a regionally $\widetilde\tau_f$-recurrent point $(x,a)\in X\times\mathbb A$, every point in $\{x\}\times\mathbb A$ is regionally $\widetilde\tau_f$-recurrent.
The minimality of $(X,T)$ and Theorem 7.13 in \cite{G-H} imply that every point in $X\times \mathbb A$ is regionally $\widetilde\tau_f$-recurrent, hence $f(\tau,x)$ is recurrent.

A cocycle $f(n,x)$ of a $\Z$-flow is topologically recurrent if and only if $\widetilde T_f$ is topologically conservative, i.e. for every open neighbourhood $U\subset X\times\mathbb A$ there exists an integer $n\neq 0$ so that $\widetilde T_f^n(U)\cap U\neq\emptyset$.
\end{remarks}

One of the most important concepts in the study of cocycles is the essential range, originally introduced in the measure theoretic category by Schmidt \cite{Sch}.

\begin{defi}\label{def:er}
Let $f(\tau,x)$ be a cocycle of a minimal compact metric flow $(X,T)$.
An element $a\in \mathbb A$ is in the set $E(f)$ of \emph{topological essential values} if for arbitrary neighbourhoods $\mc U\subset X$ and $U(a)\subset \mathbb A$ of $a$ there exists an element $\tau\in T$ so that
\begin{equation*}
\mc U\cap\tau^{-1}(\mc U)\cap\{x\in X: f(\tau,x)\in U(a)\}
\end{equation*}
is non-empty.
The set $E(f)$ is also called the \emph{topological essential range}.
The cocycle identity implies that $f(\mathbf 1_T,x)=\mathbf 0_\mathbb A$ for all $x\in X$ and hence $\mathbf 0_\mathbb A\in E(f)$.
Moreover, the essential range is always a closed \emph{subgroup} of $\mathbb A$ (cf. \cite{LM}, Proposition 3.1, which carries over from the case of a minimal $\Z$-action to a general Abelian group acting minimally).
\end{defi}

\begin{fact}\label{fact:er}
If $f(\tau,x)$ is a cocycle with full topological essential range $E(f)=\mathbb A$, then $\mc D_{T,f}(x,a)\subset\{x\}\times\mathbb A$ holds for every $(x,a)\in X\times\mathbb A$.
By Fact \ref{fact:o_p} there exists a $T$-invariant dense $G_\delta$ set $\mc F\subset X$ with $\{x\}\times \mathbb A\subset\bar{\mc O}_{T,f}(x,a)$ for every $(x,a)\in\mc F\times \mathbb A$.
For every $\tau\in T$ follows that $\{\tau x\}\times \mathbb A\subset\bar{\mc O}_{T,f}(x,g)$, and by the minimality of the flow $(X,T)$ every $(x,a)\in\mc F\times \mathbb A$ is a transitive point for $\widetilde\tau_f$.
\end{fact}

Throughout this note we shall use a notion of ``relative'' triviality of cocycles.

\begin{defi}\label{def:rnt}
Let $f_1(\tau,x)$ and $f_2(\tau,x)$ be $\R$-valued cocycles of a minimal compact metric flow $(X,T)$.
We shall call the cocycle $f_2(\tau,x)$ \emph{relatively trivial} with respect to $f_1(\tau,x)$, if for every sequence $\{(\tau_k,x_k)\}_{k\geq 1}\subset T\times X$ with $d(x_k, \tau_k x_k)\to 0$ and $f_1(\tau_k,x_k)\to 0$ it holds also that $f_2(\tau_k,x_k)\to 0$ as $k\to\infty$.
For a sequence $\{\tau_k\}_{k\geq 1}\subset T$ and a point $\bar x\in X$ so that $\tau_k\bar x$ and $f_1(\tau_k,\bar x)$ are convergent, this implies that also $f_2(\tau_k,\bar x)$ is convergent.
\end{defi}

By the following lemma it suffices to verify an essential value condition ``locally''.

\begin{lemma}\label{lem:trans}
Let $(X,T)$ be a minimal compact metric flow with an \emph{Abelian} group $T$ acting, and let $f(\tau,x)$ be a cocycle of $(X,T)$ with values in an Abelian l.c.s. group $\mathbb A$.
If there exists a sequence $\{(\tau_k,x_k)\}_{k\geq 1}\subset T\times X$ with $d(x_k, \tau_k x_k)\to 0$ and $f(\tau_k,x_k)\to a\in \mathbb A_\infty$ ($\R_\infty\times\R_\infty$ for $\mathbb A=\R^2$, respectively) as $k\to\infty$, then for every $x\in X$ it holds that $(x,a)\in\mc D_{T,f}(x,\mathbf 0_\mathbb A)$.
Hence if $a\in \mathbb A$ is finite, then $a\in E(f)$.

Now let $g=(g_1,g_2):T\times X\longrightarrow\R^2$ be a cocycle of the flow $(X,T)$ with a sequence $\{(\tau_k,x_k)\}_{k\geq 1}\subset T\times X$ so that $d(x_k,\tau_k x_k)\to 0$, $g_1(\tau_k,x_k)\to 0$, and $g_2(\tau_k,x_k)\nrightarrow 0$ as $k\to\infty$.
Then there exist a point $\bar x\in X$ and a sequence $\{\bar\tau_k\}_{k\geq 1}\subset T$ so that $d(\bar x,\bar\tau_k \bar x)\to 0$ and $g(\bar\tau_k,\bar x)\to(0,\infty)$ as $k\to\infty$.
Moreover, for an extension $(Y,T)$ of $(X,T)=\pi(Y,T)$ there exists a sequence $\{(\tilde\tau_k,y_k)\}_{k\geq 1}\subset T\times Y$ with $d_Y(y_k,\tilde\tau_k y_k)\to 0$ and $(g\circ\pi)(\tilde\tau_k,y_k)\to(0,\infty)$ as $k\to\infty$.
\end{lemma}

\begin{proof}
We let $\{(\tau_k,x_k)\}_{k\geq 1}\subset T\times X$ be a sequence with the properties above, and we may assume that $x_k\to x'\in X$ as $k\to\infty$.
For arbitrary neighbourhoods $\mc U\subset X$ and $U(a)$ of $a\in \mathbb A_\infty$ we can fix an element $\tau\in T$ with $\tau x'\in\mc U$, and since the group $T$ is Abelian it holds that $\tau x_k\to\tau x'$ and $\tau_k\tau x_k=\tau\tau_k x_k\to \tau x'$ as $k\to\infty$.
From the cocycle identity and the continuity of $f(\tau,\cdot)$ follows
\begin{eqnarray*}
f(\tau_k,\tau x_k) & = & f(\tau,\tau_k x_k) + f(\tau_k,x_k)+f(\tau^{-1},\tau x_k)\\
& = & f(\tau,\tau_k x_k)+ f(\tau_k,x_k)-f(\tau, x_k) \to a
\end{eqnarray*}
as $k\to\infty$, and for all $k$ large enough it holds that $\tau x_k$, $\tau_k\tau x_k\in\mc U$ and $f(\tau_k,\tau x_k)\in U(a)$.
Since the neighbourhoods $\mc U$ and $U(a)$ were arbitrary, we have $(x,a)\in\mc D_{T,f}(x,\mathbf 0_\mathbb A)$ for every $x\in X$ and $a\in E(f)$ if $a\neq\infty$.

If $g(\tau_k,x_k)\nrightarrow(0,\infty)$ as $k\to\infty$, then $E(g)$ has an element $(0,c)$ with $c\in\R\setminus\{0\}$.
Since $E(g)$ is a closed subspace of $\R^2$, we can start over with a sequence $\{(\tau_k,x_k)\}_{k\geq 1}\subset T\times X$ so that $d(x_k,\tau_k x_k)\to 0$, $g_1(\tau_k,x_k)\to 0$, and $|g_2(\tau_k,x_k)|\to\infty$ as $k\to\infty$.
The statement above implies that $(x,0,\infty)\in\mc D_{T,g}(x,0,0)$ for every $x\in X$, and by Fact \ref{fact:o_p} we can select $\bar x\in X$ and a sequence $\{\bar\tau_k\}_{k\geq 1}\subset T$ so that $\bar\tau_k\bar x\to\bar x$ and $g(\bar\tau_k,\bar x)\to(0,\infty)$.
For an arbitrary point $\bar y\in\pi^{-1}(\bar x)$ we can select an increasing sequence of positive integers $\{k_l\}_{l\geq 1}$ with $d_Y(\bar\tau_{k_{l+1}} \bar y,\bar\tau_{k_l}\bar y)\to 0$ and $(g\circ\sigma)(\bar\tau_{k_{l+1}}(\bar\tau_{k_l})^{-1},\bar\tau_{k_l}\bar y)\to(0,\infty)$ and put $\{(\tilde\tau_l,y_l)=(\bar\tau_{k_{l+1}}(\bar\tau_{k_l})^{-1},\bar\tau_{k_l}\bar y)\}_{l\geq 1}$.
\end{proof}

\begin{defi}
Let $f(\tau,x)$ be a cocycle of a minimal compact metric flow $(X,T)$ with values in an Abelian l.c.s. group $\mathbb A$, and let $b:X\longrightarrow \mathbb A$ be a continuous function.
Another cocycle of the flow $(X,T)$ can be defined by the $\mathbb A$-valued function
\begin{equation*}
g(\tau,x)=f(\tau,x)+b(\tau x)-b(x).
\end{equation*}
The cocycle $g(\tau,x)$ is called \emph{topologically cohomologous} to the cocycle $f(\tau,x)$ with the \emph{transfer function} $b(x)$.
A cocycle $g(\tau,x)=b(\tau x)-b(x)$ topologically cohomologous to zero is bounded on $T\times X$ and called a \emph{topological coboundary}.
\end{defi}

The Gottschalk-Hedlund theorem (\cite{G-H}, Theorem 14.11) characterises topological coboundaries of a minimal $\Z$-action as cocycles bounded on at least one semi-orbit.
The generalisation to an Abelian group $T$ acting minimally is natural.

\begin{fact}\label{fact:GH}
A real valued topological cocycle $f(\tau,x)$ of a minimal compact metric flow $(X,T)$ with an Abelian acting group $T$ is a coboundary if and only if there exists a point $\bar x\in X$ so that the function $\tau\mapsto f(\tau,\bar x)$ is bounded on $T$.
For the groups $T=\Z$ and $T=\R$ acting, the boundedness on a semi-orbit is sufficient.

A real valued cocycle $f(\tau,x)$ is also a topological coboundary if for every sequence $\{(\tau_k,x_k)\}_{k\geq 1}\subset T\times X$ with $d(x_k,\tau_k x_k)\to 0$ the set $\{f(\tau_k,x_k)\}_{k\geq 1}\subset\R$ is bounded.
\end{fact}

\begin{proof}
Suppose that $\tau\mapsto f(\tau,\bar x)$ is bounded on $T$.
By the cocycle identity holds
\begin{equation*}
f(\tau,\tau'\bar x)=f(\tau\tau',\bar x)-f(\tau',\bar x)
\end{equation*}
for all $\tau,\tau'\in T$, and by the density of the $T$-orbit of $\bar x$ follows the boundedness of $f(\tau,x)$ on $T\times X$ and thus the triviality of the subgroup $E(f)=\{0\}$.
By the density of the $T$-orbit of $\bar x$ and the boundedness of $\tau\mapsto f(\tau,\bar x)$, the intersection $\{x\}\times\R\cap\bar{\mc O}_{T,f}(\bar x,0)$ is non-empty for every $x\in X$.
For every $x\in X$ this intersection is a singleton, since otherwise Lemma \ref{lem:trans} proves a non-zero element in $E(f)$.
Hence the compact set $\bar{\mc O}_{T,f}(\bar x,0)$ is the graph of a continuous function $b:X\longrightarrow\R$ with $f(\tau,\bar x)=b(\tau\bar x)$, and thus $f(\tau,x)=b(\tau x)-b(x)$ holds for every $(\tau,x)\in T\times X$.
For $T=\Z$ and $T=\R$ the set of \emph{limit points} of a semi-orbit is a $T$-invariant closed subset of $X$, which is non-empty by compactness and equal to $X$ by minimality.
We can conclude the proof as above, but using the semi-orbit.

Now suppose that $f(\tau,x)$ is not a topological coboundary and let $\bar x\in X$ be arbitrary.
Then there exists a sequence $\{\tau'_l\}_{l\geq 1}\subset T$ with $|f(\tau'_l,\bar x)|\to\infty$, and we may assume that $\tau'_l \bar x\to x'$ as $l\to\infty$.
Since $(X,T)$ is minimal, there exists sequence $\{\tau''_k\}_{k\geq 1}\subset T$ with $\tau''_k x'\to\bar x$ as $k\to\infty$.
A diagonalisation with a sufficiently increasing sequence of positive integers $\{l_k\}_{l\geq 1}$ yields for $\tau_k=\tau_k''\tau_{l_k}'$ that $\tau_k\bar x\to\bar x$ and $|f(\tau_k,\bar x)|=|f(\tau_k'',\tau_{l_k}'\bar x)+f(\tau_{l_k}',\bar x)|\to\infty$ as $k\to\infty$.
\end{proof}

The following lemma appeared originally in the paper \cite{A} in a setting for $\R^d$-valued cocycles of a minimal rotation on a torus.

\begin{lemma}\label{lem:at}
Let $f(\tau,x)$ be a real valued topological cocycle of a minimal compact metric flow $(X,T)$ with an Abelian acting group $T$.
If the skew product action $\widetilde\tau_f$ is \emph{not} point transitive on $X\times\R$, then for every neighbourhood $U\subset\R$ of $0$ there exist a compact symmetric neighbourhood $K\subset U$ of $0$ and an $\varepsilon>0$ so that for every $\tau\in T$ holds
\begin{equation}\label{eq:s_l}
\{x\in X:d(x,\tau x)<\varepsilon\enspace\textup{and}\enspace f(\tau,x)\in 2K\setminus K^0\}=\emptyset .
\end{equation}
\end{lemma}

\begin{proof}
Suppose that $f(\tau,x)$ is real valued and $\widetilde\tau_f$ is not point transitive.
By Fact \ref{fact:er} the essential range $E(f)$ is a proper closed subgroup of $\R$, and thus there exists a compact symmetric neighbourhood $K\subset U$ of $0$ with $(2K\setminus K^0)\cap E(f)=\emptyset$.
If the assertion is false for the neighbourhood $K$, then there exists a sequence $\{(\tau_k,x_k)\}_{k\geq 1}\subset T\times X$ with $d(x_k, \tau_k x_k)\to 0$ and $f(\tau_k,x_k)\to t\in 2K\setminus K^0$.
Now Lemma \ref{lem:trans} implies $t\in E(f)\cap 2K\setminus K^0$, in contradiction to the choice of $K$.
\end{proof}

We shall commence the study of cocycles of distal minimal flows by the generalisation of the results for minimal rotations in \cite{G-H} and \cite{LM}.

\begin{proposition}\label{prop:isom}
Let $(X,T)$ be a minimal compact \emph{isometric} flow with a compactly generated Abelian acting group $T$, and let $f(\tau,x)$ be a topologically recurrent real valued cocycle of $(X,T)$.
Then the cocycle $f(\tau,x)$ is either a coboundary or its skew product extension $\widetilde \tau_f$ is point transitive on $X\times\R$.
\end{proposition}

\begin{proof}
Suppose that the cocycle $f(\tau,x)$ is not a coboundary and $\widetilde \tau_f$ is not point transitive.
Then by Lemma \ref{lem:at} there exist a compact symmetric neighbourhood $K$ of $0$ and an $\varepsilon>0$ so that equality (\ref{eq:s_l}) holds for every $\tau\in T$.
Furthermore, if $L\subset T$ is a compact generative subset, then $\varepsilon>0$ can be chosen small enough so that for all $\tau'\in L$ and $x,x'\in X$ with $d(x,x')<\varepsilon$ it holds that
\begin{equation*}
f(\tau',x)-f(\tau',x')\in K^0 .
\end{equation*}
By Fact \ref{fact:GH} we can fix a pair $(\bar\tau,\bar x)\in T\times X$ with $d(\bar x,\bar\tau\bar x)<\varepsilon$ and $f(\bar\tau,\bar x)\notin 2K$, since $f(\tau,x)$ is not a coboundary.
The Abelian group $T$ acts on $X$ isometrically, and thus $d(\bar x,\bar\tau\bar x)<\varepsilon$ implies that $d(\tau'\bar x,\bar\tau\tau' \bar x)=d(\tau'\bar x,\tau'\bar\tau\bar x)<\varepsilon$.
Together with equality (\ref{eq:s_l}) we can conclude for every $\tau'\in L$ that
\begin{equation*}
f(\bar\tau,\tau'\bar x)=f(\bar\tau,\bar x)-f(\tau',\bar x)+f(\tau',\bar\tau\bar x)\notin 2K ,
\end{equation*}
and hence both of the real numbers $f(\bar\tau,\bar x)$ and $f(\bar\tau,\tau'\bar x)$ are elements of the one and the same of the disjoint sets $\R^+\setminus 2K$ and $\R^-\setminus 2K$.
Since the set $L$ is generative in the Abelian group $T$ acting minimally on $X$, it follows by induction that $f(\bar\tau, x)$ is in the closure of one of the sets $\R^+\setminus 2K$ and $\R^-\setminus 2K$ for every $x\in X$.
Thus we have a constant $c>0$ with $|f(\bar\tau^k,x)|> |k| c$ for every integer $k$, and we define a subset $P\subset T$ by
\begin{equation*}
P=\cup_{k\geq 1}\bar\tau^k\cdot\{\tau\in T:f(\tau,\cdot)<|k| c/2\} .
\end{equation*}
Given two integers $k,k'\geq 1$ and $\bar\tau^k\tau$, $\bar\tau^{k'}\tau'\in P$ with $f(\tau,\cdot)<|k| c/2$ and $f(\tau',\cdot)<|k'| c/2$, we can conclude that $\bar\tau^k\tau\bar\tau^{k'}\tau'=\bar\tau^{k+k'}(\tau\tau')$ with $f(\tau\tau',\cdot)<|k+k'| c/2$, hence $P$ is a semigroup.
Moreover, the semigroup $P$ contains a translate of every compact set $L\subset T$, since for large enough $k\geq 1$ the inequality $f(\tau,x)<|k| c/2$ holds for every $\tau\in L$ and every $x\in X$.
Therefore $P$ is a replete semigroup in $T$ so that $|f(\tau,x)|>c/2$ holds for every $(\tau,x)\in P\times X$, which contradicts the existence of a dense $G_\delta$ set of $\widetilde\tau_f$-recurrent points (cf. Remarks \ref{rems:rec}).
\end{proof}

The \emph{Rokhlin extensions} and the \emph{Rokhlin skew products} have been studied in the measure theoretic setting in \cite{LL} and \cite{LP}.
We shall introduce the notion of a \emph{perturbed Rokhlin skew product}, which will be inevitable in our main result.

\begin{defi}
Suppose that $(X,T)$ is a distal minimal compact metric flow and $(M,\{\phi^t:t\in\R\})$ is a distal minimal compact metric $\R$-flow.
Let $f:T\times X\longrightarrow\R$ be a cocycle of $(X,T)$ with a point transitive skew product $\widetilde\tau_f$ on $X\times\R$.
We define the \emph{Rokhlin extension} $\tau_{\phi,f}$ on $X\times M$ by
\begin{equation*}
\tau_{\phi,f}(x,m)=(\tau x,\phi^{f(\tau,x)}(m)) ,
\end{equation*}
which is an action of the group $T$ on $X\times M$ due to the cocycle identity for $f(\tau,x)$.
If $(\bar x,0)$ is a transitive point for $\widetilde\tau_f$, then by the minimality of $(M,\{\phi^t:t\in\R\})$ every point $(\bar x,m)$ with $m\in M$ is a transitive point for $(X\times M,T)$.
Since the flow $(X\times M,T)$ is distal by the distality of its components, it is even minimal.

The skew product extension of $(X\times M,T)$ by the cocycle $(\tau,x,m)\mapsto f(\tau,x)$ is the \emph{Rokhlin skew product} $\widetilde\tau_{\phi,f}$ on $X\times M\times \R$ with
\begin{equation*}
\widetilde\tau_{\phi,f}(x,m,t)=(\tau x,\phi^{f(\tau,x)}(m),t+f(\tau,x)) .
\end{equation*}
Let $g:\R\times M\longrightarrow\R$ be a cocycle of the flow $(M,\{\phi^t:t\in\R\})$.
The $\R$-valued map
\begin{equation*}
(\tau,x,m)\mapsto f(\tau,x)+g(f(\tau,x),m)
\end{equation*}
defined on $T\times X\times M$ turns out to be a cocycle of the flow $(X\times M,T)$ due to the cocycle identity for $g(t,m)$.
The skew product extension of this cocycle with
\begin{equation*}
\widetilde\tau_{\phi,f,g}(x,m,t)=(\tau x,\phi^{f(\tau,x)}(m),t+f(\tau,x)+g(f(\tau,x),m)) .
\end{equation*}
is called a \emph{perturbed Rokhlin skew product} $\widetilde\tau_{\phi,f,g}$ on $X\times M\times\R$.
\end{defi}

We shall present at first the basic example of a topological Rokhlin skew product of topological type $III_0$, i.e. recurrent with a trivial topological essential range but not a topological coboundary.

\begin{example}\label{ex:zi}
Let $f:\T\longrightarrow\R$ be a continuous function with a point transitive skew product extension $\widetilde T_f$ of the irrational rotation $T$ by $\alpha$ on the torus, and let $\beta\in(0,1)$ be irrational so that the $\R$-flow $(\T^2,\{\phi^t:t\in\R\})$ defined by $\phi^t(y,z)=(y+t,z+\beta t)$ is minimal and distal.
The minimal and distal Rokhlin extension $T_{\phi,f}$ on $\T^3$ is
\begin{equation*}
T_{\phi,f}(x,y,z)=(x+\alpha,y+f(x),z+\beta f(x)),
\end{equation*}
and putting $h(x,y,z)=f(x)$ for all $(x,y,z)\in\T^3$ gives a topological type $III_0$ cocycle $h(n,(x,y,z))$ of the homeomorphism $T_{\phi,f}$ with the skew product extension $\widetilde T_{\phi,f}$.
Indeed, since $\widetilde T_f$ is point transitive, the cocycle $h(n,(x,y,z))$ is recurrent, but it is not bounded and therefore no topological coboundary.
Furthermore, a sequence $\{t_n\}_{n\geq 1}\subset\R$ with $t_n\mod 1\to 0$ and $(\beta t_n)\mod 1\to 0$ cannot have a finite cluster point apart from zero, and hence $E(h)=\{0\}$.
For a point $\bar x\in\T$ so that $(\bar x,0)\in\T\times\R$ is transitive under $\widetilde T_f$ and arbitrary $y,z\in\T$ the orbit closure of $(\bar x,y,z,0)$ under the skew product extension of $T_{\phi,f}$ by $h$ is of the form
\begin{equation*}
\bar{\mc O}_{\widetilde T_{\phi,f}}((\bar x,y,z),0)=\bar{\mc O}_{T_{\phi,f},h}((\bar x,y,z),0)=\T\times\{(\phi^t(y,z),t)\in\T^2\times\R:t\in\R\} .
\end{equation*}
The collection of these sets is a partition of $\T^3\times\R$ into $\widetilde T_{\phi,f}$-orbit closures.
\end{example}

The next example makes clear that the perturbation of a Rokhlin skew product by a cocycle is an essential component, which in general cannot be eliminated by continuous cohomology.

\begin{example}\label{ex:pe}
Let $T$, $f$, $h$, and $\{\phi^t:t\in\R\}$ be defined as in Example \ref{ex:zi}, and suppose that $g(t,(y,z))$ is a point transitive $\R$-valued cocycle of the flow $\{\phi^t:t\in\R\}$.
From the unique ergodicity of the flow $\{\phi^t:t\in\R\}$ follows $\int_{\T^2} g(t,(y,z)) d\lambda(y,z)=0$ for every $t\in\R$, and after rescaling $g$ we can assume that $|g(t,(y,z))|<|t|/2$ for every $t\in\R$ and $(y,z)\in\T^2$.
We define a function
\begin{equation*}
\bar h(x,y,z)= f(x)+g(f(x),(y,z))
\end{equation*}
so that the cocycle of $T_{\phi,f}$ is $\bar h(n,(x,y,z))= f(n,x)+g(f(n,x),(y,z))$ for every $n$.
Since the perturbation $g(f(n,x),(y,z))$ is unbounded, there cannot be a continuous transfer function defined on $\T^3$ so that $\bar h$ and $h$ are cohomologous.
However, due to the condition $|g(t,(y,z))|<|t|/2$ the set
\begin{equation*}
\T\times\{(\phi^t(y,z),t+g(t,(y,z)))\in\T^2\times\R:t\in\R\}
\end{equation*}
is closed, and it coincides with $\bar{\mc O}_{\widetilde T_{\phi,f,g}}((\bar x,y,z),0)$ if the point $(\bar x,0)\in\T\times\R$ is transitive under $\widetilde T_f$.
Thus the structure of the orbit closures is preserved as well as these sets provide a partition of $\T^3\times\R$.
\end{example}

\begin{remark}
The structure of Example \ref{ex:zi} can be revealed from the toral extensions of $T_{\phi,f}$ by the function $(\gamma h) \mod 1$ for all $\gamma\in\R$.
This distal homeomorphism of $\T^4$ is transitive and hence minimal for rationally independent $1$, $\beta$, and $\gamma$.
However, for $\gamma=1$ and $\gamma=\beta$ the orbit closures collapse to graphs representing the dependence of $h$ and the action on the coordinates of the torus.
The same approach will not be successful with respect to Example \ref{ex:pe}.
It can be verified that for every $\gamma\in\R$ the toral extension of $T_{\phi,f}$ by the function $(\gamma\bar h) \mod 1$ is minimal on $\T^4$.
\end{remark}

The main result of this note puts these examples into a structure theorem.

\begin{strth*}
Suppose that $(X,T)$ is a distal minimal compact metric flow with a compactly generated Abelian acting group $T$ and that $f:T\times X\longrightarrow\R$ is a topologically recurrent cocycle which is not a coboundary.
Then there exist a factor $(X_\alpha,T)=\pi_\alpha(X,T)$, a topological cocycle $f_\alpha:T\times X_\alpha\longrightarrow\R$ of $(X_\alpha,T)$, and a distal minimal compact metric $\R$-flow $(M,\{\phi^t:t\in\R\})$, so that the Rokhlin extension $(X_\alpha\times M,T)$ with the action $\tau_{\phi,f_\alpha}$ is a factor $(Y,T)=\pi_Y(X,T)$ of $(X,T)$.
The cocycle $f(\tau,x)$ is topologically cohomologous to $(f_Y\circ\pi_Y)(\tau,x)=f(\tau,x)+b(\tau x)-b(x)$ with a topological cocycle $f_Y:T\times Y\longrightarrow\R$ of the flow $(Y,T)$ so that
\begin{equation}\label{eq:fpy}
\mc D_{T, f_Y\circ\pi_Y}(x,0)\cap(\pi_\alpha^{-1}(\pi_\alpha(x))\times\{0\})=\pi_Y^{-1}(\pi_Y(x))\times\{0\}
\end{equation}
holds for all $x\in X$.
Moreover, there exists a cocycle $g:\R\times M\longrightarrow\R$ of the $\R$-flow $(M,\{\phi^t:t\in\R\})$ so that the cocycle $(\mathbbm 1+g)(t,m)= t+g(t,m)$ is topologically \emph{transient} and
\begin{equation}\label{eq:f_Y2}
f_Y(\tau,(x,m))=f_\alpha(\tau,x)+g(f_\alpha(\tau,x),m)=(\mathbbm 1+g)(f_\alpha(\tau,x),m)
\end{equation}
holds for every $\tau\in T$ and $(x,m)\in Y=X_\alpha\times M$.
Thus the skew product $\widetilde\tau_{f_Y}$ on $Y\times\R$ is the perturbed Rokhlin skew product $\widetilde \tau_{\phi,f_\alpha, g}$.
\end{strth*}

\noindent We shall conclude the proof of this theorem in the next section of this note.
\medskip

The application of the structure theorem for a topological ergodic decomposition requires a suitable topology on the hyperspace of the non-compact space $X\times\R$.
We shall use the Fell topology on the hyperspace of \emph{non-empty} closed subsets of a locally compact separable metric space.
Given finitely many open neighbourhoods $\mc U_1,\dots,\mc U_k$ and a compact set $K$, an element of the Fell topology base consists of all non-empty closed subsets which intersect each of the open neighbourhoods $\mc U_1,\dots,\mc U_k$ while being disjoint from $K$.
This topology is separable, metrisable, and $\sigma$-compact (cf. \cite{HLP}).
The Fell topology was introduced in \cite{Fe} as a compact topology on the hyperspace of all closed subsets, with the empty set as infinity.

\begin{decth*}
Suppose that $f:T\times X\longrightarrow\R$ is a topologically recurrent cocycle of a distal minimal compact metric flow $(X,T)$ with a compactly generated Abelian acting group $T$.
The prolongations $\mc D_{T,f}(x,s)\subset X\times\R$ of the skew product action $\widetilde\tau_{f}$ with $(x,s)\in X\times\R$ define a \emph{partition} of $X\times\R$.
The mapping $(x,s)\mapsto\mc D_{T,f}(x,s)$ is continuous with respect to the Fell topology on the hyperspace of non-empty closed subsets of $X\times\R$, and the right translation on $X\times\R$ is a minimal continuous $\R$-action on the set of prolongations.
If the cocycle $f(\tau,x)$ is not a topological coboundary, then the set of all prolongations in the skew product is Fell compact.
\end{decth*}

\begin{tmath*}
A recurrent cocycle $f(\tau,x)$ apart from a coboundary has a minimal compact metric flow as a topological version of the \emph{Mackey action}.
Its phase space is the set of prolongations in the skew product with the Fell topology, with the right translation of $\R$ acting on the prolongations.
This flow is a distal extension (possibly the trivial extension) of a weakly mixing compact metric flow (possibly the trivial flow).
The Mackey action is distal if and only if the perturbation cocycle $g(t,m)$ in the structure theorem is a topological coboundary.
\end{tmath*}

While most of the properties of the topological Mackey action are part of the decomposition theorem, its structure as a distal extension of a weakly mixing flow will be verified in the next section of this note.
The proof of the decomposition theorem depends on the following general lemma on transient cocycles of minimal $\R$-flows, which might be of independent interest.

\begin{lemma}\label{lem:tr_coc}
Let $(M,\{\phi^t:t\in\R\})$ be a minimal compact metric $\R$-flow and let $h(t,m):\R\times M\longrightarrow\R$ be a \emph{transient} cocycle of $(M,\{\phi^t:t\in\R\})$.
Then for every point $(m,s)\in M\times\R$ the orbit $\mc O_{\phi,h}(m,s)$, the orbit closure $\bar{\mc O}_{\phi,h}(m,s)$, and the prolongation $\mc D_{\phi,h}(m,s)$ under the skew product extension $\widetilde{\phi^t}_h$ coincide.
The mapping from points to their orbits in $M\times\R$ is continuous with a compact range with respect to the Fell topology, and the right translation on $M\times\R$ provides a minimal continuous $\R$-action on the set of orbits.
Moreover, for every $m\in M$ the mapping $t\mapsto h(t,m)$ maps $\R$ \emph{onto} $\R$.
\end{lemma}

\begin{proof}
Since prolongations are closed sets, it suffices to verify that for every $(m,s)\in M\times\R$ the orbit and the prolongation coincide.
Otherwise, there exist two points $(m,s), (m',s')\in M\times\R$ so that $(m',s')$ is not in the $\widetilde{\phi^t}_h$-orbit of $(m,s)$, however there exists a sequence $\{(t_k,m_k)\}_{k\geq 1}\subset\R\times M$ so that $(t_k,m_k)\to (+\infty,m)$ and
\begin{equation*}
\widetilde{\phi^{t_k}}_h(m_k,s)=(\phi^{t_k}(m_k),s+h(t_k,m_k))\to (m',s') .
\end{equation*}
If there exists a compact set $L\subset\R$ with $h([0,t_k],m_k)\subset L$ for all $k\geq 1$, then $h([0,\infty),m)\subset L$ since $m_k\to m$, and by Fact \ref{fact:GH} the cocycle $h(t,m)$ is a coboundary in contradiction to its transience.
Therefore we have an increasing sequence of integers $\{k_l\}_{l\geq 1}$, a sequence $\{t'_l\}_{l\geq 1}\subset\R$ with $t'_l\in[0,t_{k_l}]$, and $S\in\{+1,-1\}$ so that
\begin{equation*}
S\cdot h(t'_l,m_{k_l})=\max_{t\in[0,t_{k_l}]}S\cdot h(t,m_{k_l})\to +\infty
\end{equation*}
as $l\to\infty$.
For every limit point $\bar m$ of the sequence $\{\phi^{t'_l}(m_{k_l})\}_{l\geq 1}$ it holds that $S\cdot h(t,\bar m)\leq 0$ for all $t\in\R$, and the mapping $t\mapsto h(t,\bar m)$ maps each of the sets $\R^+$ and $\R^-$ onto $S\cdot \R^-$.
Hence for every $t\in\R^+$ there exists a $t'\in\R^-$ with $h(t,\bar m)=h(t',\bar m)$, and by the density of the semi-orbit $\{\phi^t(\bar m):t\in\R^+\}$ (cf. the proof of Fact \ref{fact:GH}) and the cocycle identity the open set
\begin{equation*}
M_k=\{m\in M:|h(t,m)|<2^{-k}\enspace\textup{for some}\enspace t<-k\}
\end{equation*}
is dense for every integer $k\geq 1$.
For a point $m_k$ in the dense $G_\delta$ set $\bigcap_{t\in\Q}\phi^t(M_k)$, we can find rational numbers $t_1,\dots,t_k<-k$ so that $\phi^{t_1+\cdots+t_l}(m_k)\in M_k$ and $|h(t_1+\dots+t_l,m_k)|<k 2^{-k}$ for all $1\leq l\leq k$.
Since $M$ is compact, there exists a point $\tilde m\in M$ each neighbourhood of which contains at least two distinct points out of the finite sequence $\phi^{t_1}(m_k),\dots,\phi^{t_1+\cdots+t_k}(m_k)\in M_k$ for an infinite set of integers $k\geq 1$.
Thus $(\tilde m,0)$ is a regionally recurrent point for the skew product $\widetilde{\phi^t}_h$, in contradiction to the Remarks \ref{rems:rec} and the transience of the cocycle $h(t,m)$.

Suppose for some $m\in M$ the mapping $t\mapsto h(t,m)$ is not onto $\R$, then either there exists a compact set $L\subset\R$ so that one of the inclusions $h([0,\infty),m)\subset L$, $h((-\infty,0],m)\subset L$ holds, or there exists a point $\bar m\in M$ as above.
In each of these cases, we can obtain a contradiction to the transience of $h(t,m)$ as above.

Let $(m,s)\in M\times\R$ be arbitrary, and let a Fell neighbourhood of its \emph{closed} $\widetilde{\phi^t}_h$-orbit be defined by the open neighbourhoods $\mc U_1,\dots,\mc U_k$ and a compact set $K$.
Obviously the $\widetilde{\phi^t}_h$-orbit of every point in a suitable neighbourhood of $(m,s)$ intersects each of the neighbourhoods $\mc U_1,\dots,\mc U_k$.
Moreover, the $\widetilde{\phi^t}_h$-orbit is disjoint from $K$ for every point in a suitable neighbourhood of $(m,s)$, since otherwise the $\widetilde{\phi^t}_h$-prolongation of $(m,s)$ intersects the compact set $K$.
This contradicts the coincidence of orbits and prolongations of $\widetilde{\phi^t}_h$, and hence the mapping of a point $(m,s)\in M\times\R$ to its $\widetilde{\phi^t}_h$-orbit is Fell continuous.
By the surjectivity of the mapping $t\mapsto h(t,m)$, every $\widetilde{\phi^t}_h$-orbit intersects the set $\{(m,0):m\in M\}$.
The Fell compactness of the set of $\widetilde{\phi^t}_h$-orbits in $M\times\R$ follows, and the right translation on $M\times\R$ is a Fell continuous $\R$-action by the definition of the Fell neighbourhoods.
Since the points $(\phi^t(m),0)$ and $(m,-h(m,t))$ are within the same $\widetilde{\phi^t}_h$-orbit, the minimality of $(M,\{\phi^t:t\in\R\})$ implies the minimality of the right translation action.
\end{proof}

\begin{proof}[Proof of the decomposition theorem]
For a topological coboundary $f(\tau,x)$ the assertions are trivial, since the prolongations are just the right translates of the graph of the transfer function.
If the cocycle $f(\tau,x)$ is not a coboundary, then we apply the structure theorem.
For a point $\bar x_\alpha\in X_\alpha$ with $\bar{\mc O}_{T,f_\alpha}(\bar x_\alpha,0)=X_\alpha\times\R$ and every $m\in M$ follows from equality (\ref{eq:f_Y2}) that $X_\alpha\times\mc O_{\phi,\mathbbm 1+g}(m,0)\subset\bar{\mc O}_{T,f_Y}(\bar x_\alpha,m,0)\subset X_\alpha\times\bar{\mc O}_{\phi,\mathbbm 1+g}(m,0)$, and  by Lemma \ref{lem:tr_coc} holds $\mc O_{\phi,\mathbbm 1+g}(m,0)=\bar{\mc O}_{\phi,\mathbbm 1+g}(m,0)=\mc D_{\phi,\mathbbm 1+g}(m,0)$.
The set of points $\bar x_\alpha\in X_\alpha$ with $\bar{\mc O}_{T,f_\alpha}(\bar x_\alpha,0)=X_\alpha\times\R$ is a dense $G_\delta$ set, and for arbitrary $(x_\alpha,m)\in X_\alpha\times M$ and $(x_\alpha',m',t')\in\mc D_{T,f_Y}(x_\alpha,m,0)$ we can find a sequence $\{(\tau_k,\bar x_k,m_k)\}\subset T\times X_\alpha\times M$ with $(\bar x_k,m_k)\to (x_\alpha,m)$, $\tau_k(\bar x_k,m_k)\to (x_\alpha',m')$, $f_Y(\tau_k,(\bar x_k, m_k))\to t'$, and $X_\alpha\times\mc O_{\phi,\mathbbm 1+g}(m_k,0)=\bar{\mc O}_{T,f_Y}(\bar x_k,m_k,0)$.
Therefore $(m',t')\in\mc D_{\phi,\mathbbm 1+g}(m,0)$ so that $\mc D_{T,f_Y}(x_\alpha,m,0)\subset X_\alpha\times\mc D_{\phi,\mathbbm 1+g}(m,0)$.
We conclude for every $(x_\alpha,m,s)\in X_\alpha\times M\times\R$ that
\begin{equation*}
\mc D_{T,f_Y}(x_\alpha,m,s)=X_\alpha\times\{(\phi^t(m),s+t+g(t,m)):t\in\R\} ,
\end{equation*}
and these sets define a partition of $X_\alpha\times M\times\R$.
The Fell continuity of the mapping $(x_\alpha,m,s)\mapsto\mc D_{T,f_Y}(x_\alpha,m,s)$, the compactness of its range, and the minimality of the right translation follow directly from Lemma \ref{lem:tr_coc}.

Suppose that $(x,s)\in X\times\R$ and $(x_\alpha',m',s')\in\mc D_{T,f_Y}(\pi_Y(x),s)$, and let $\{\tau_k\}_{k\geq 1}\subset T$ and $\{y_k\}_{k\geq 1}\subset X_\alpha\times M$ be sequences with $y_k\to\pi_Y(x)$, $\tau_k y_k\to (x_\alpha',m')$, and $f_Y(\tau_k,y_k)\to s'-s$.
Since $\pi_Y$ is a distal homomorphism, it is an open onto mapping, and the mapping $y\mapsto\pi_Y^{-1}(y)$ is continuous with respect to the Hausdorff metric $d_H$ (cf. \cite{K}, p. 68, Theorem 1, and p. 47).
Therefore we can define a sequence $\{x_k\in\pi_Y^{-1}(y_k)\}_{k\geq 1}\subset X$ so that $x_k\to x$, $\tau_k x_k\to x'\in\pi_Y^{-1}(x'_\alpha,m')$, and $(f_Y\circ\pi_Y)(\tau_k,x_k)\to s'-s$.
By equality (\ref{eq:fpy}), for every $x''\in\pi_Y^{-1}(\pi_Y(x'))$ there exists a sequence $\{(\bar\tau_k,\bar x_k)\}_{k\geq 1}$ with $\bar x_k$ sufficiently close to $x_k$ so that $\bar x_k\to x$, $\tau_k\bar x_k\to x'$, $(f_Y\circ\pi_Y)(\tau_k,\bar x_k)\to s'-s$, $d(\bar\tau_k \tau_k\bar x_k,x'')<d_H(\pi_Y^{-1}(\pi_Y(\tau_k x_k)),\pi_Y^{-1}(\pi_Y(x')))+2^{-k}$, and $(f_Y\circ\pi_Y)(\bar\tau_k,\tau_k\bar x_k)<2^{-k}$.
From $d_H(\pi_Y^{-1}(\pi_Y(\tau_k x_k)),\pi_Y^{-1}(\pi_Y(x')))\to 0$ follows that $\bar\tau_k\tau_k\bar x_k\to x''$ and $(f_Y\circ\pi_Y)(\bar\tau_k\tau_k,\bar x_k)\to s'-s$, hence $(x'',s')\in\mc D_{T,f_Y\circ\pi_Y}(x,s)$.
Thus the $\widetilde\tau_{f_Y\circ\pi_Y}$-prolongations in $X\times\R$ are exactly the pre-images of the $\widetilde\tau_{f_Y}$-prolongations under the mapping $\pi_Y\times\textup{id}_\R$.
If a Fell neighbourhood of $\mc D_{T,f_Y\circ\pi_Y}(x,s)$ is defined by the open neighbourhoods $\mc U_1,\dots,\mc U_k$ and a compact set $K$ in $X\times\R$, then by the openness of $\pi_Y$ the $\pi_Y\times\textup{id}_\R$-images of these sets define a Fell neighbourhood of $\mc D_{T,f_Y}(\pi_Y(x),s)$.
The Fell continuity of $(x,s)\mapsto \mc D_{T,f_Y\circ\pi_Y}(x,s)$, the compactness of the range, and the minimality of the right translation follow.

Let $b:X\longrightarrow\R$ be the transfer function in the structure theorem.
The homeomorphism $(x,s)\mapsto (x,s+b(x))$ on $X\times\R$ defines a homeomorphism of the hyperspace with the Fell topology commuting with the right translation.
Since it maps $\mc D_{T,f}(x,s)$ onto $\mc D_{T,f_Y\circ\pi_Y}(x,s+b(x))$, it is a topological isomorphism of the right translation actions and all properties carry over.
\end{proof}

\begin{remarks}
Though the compact metric flow $(X,T)$ is not necessarily itself a Rokhlin extension, the existence of a real-valued recurrent non-coboundary cocycle $f(\tau,x)$ with a non-transitive skew product extension $\widetilde\tau_f$ implies the existence of a Rokhlin extension factor $(Y,T)=\pi_Y(X,T)$ with a non-trivial flow $\{\phi^t:t\in\R\}$ and a cocycle $f_Y(\tau,y)$ of the flow $(Y,T)$ with $f_Y\circ\pi_Y$ cohomologous to $f$.

The surjectivity of the mapping $t\mapsto (\mathbbm 1+g)(t,m)$ for every $m\in M$ implies that $\{f(\tau,x):\tau\in T\}$ is dense in $\R$ for all $x$ in a dense $G_\delta$ subset of $X$.

For other well-known topologies on the hyperspace like the Vietoris topology and the Hausdorff topology with respect to the product metric on $X\times\R$, the continuity of the mapping $(x,s)\mapsto\mc D_{T,f}(x,s)$ can be disproved by Example \ref{ex:pe}.
\end{remarks}

\section{Proof of the structure theorem}

Furstenberg's structure theorem for distal minimal flows shall be our main tool for studying the structure of cocycles.

\begin{defi}
Let $X$ and $Y$ be compact metric spaces, let $\pi$ be a continuous mapping from of $X$ onto $Y$, and let $M$ be a compact homogeneous metric space.
Suppose that $\rho(x_1,x_2)$ is a continuous real valued function defined on the set
\begin{equation*}
R_\pi=\{(x,x')\in X^2: \pi(x)=\pi(x')\}
\end{equation*}
so that for every $y\in Y$ the function $\rho$ is a metric on the fibre $\pi^{-1}(y)$ with an isometric mapping between $\pi^{-1}(y)$ and $M$.
Then $X$ is called an $M$-\emph{bundle} over $Y$.

Now let $(X,T)$ and $(Y,T)=\pi(X,T)$ be compact metric flows with $X$ an $M$-bundle over $Y$.
If the function $\rho$ satisfies $\rho(x,x')=\rho(\tau x,\tau x')$ for all $(x,x')\in R_\pi$ and $\tau\in T$, then $(X,T)$ is called an \emph{isometric extension} of $(Y,T)$.
\end{defi}

\begin{fact}[Furstenberg's structure theorem \cite{Fu}]
Let $(X,T)$ be a distal minimal compact metric flow.
Then there exist a countable ordinal $\eta$ and factors $(X_\xi,T)=\pi_\xi(X,T)$ for each ordinal $0\leq\xi\leq\eta$ with the following properties:
\begin{enumerate}
\item $(X_\eta,T)=(X,T)$ and $(X_0,T)$ is the trivial flow.
\item $(X_\xi,T)=\pi_\xi^{\zeta}(X_{\zeta},T)$ is a factor of $(X_{\zeta},T)$ for all ordinals $0\leq\xi<\zeta\leq\eta$.
\item For every ordinal $0\leq\xi<\eta$ the flow $(X_{\xi+1},T)$ is an isometric extension of $(X_\xi,T_\xi)$.
\item For a limit ordinal $0<\xi\leq\eta$ the flow $(X_\xi,T)$ is the inverse limit of the flows $\{(X_{\zeta},T):0\leq\zeta<\xi\}$.
\end{enumerate}
A system $\{(X_\xi,T):0\leq\xi\leq\eta\}$ with the properties above is called a quasi-isometric system or \emph{I}-system.
\end{fact}

\begin{defi}
An \emph{I}-system $\{(X_\xi,T):0\leq\xi\leq\eta\}$ is called \emph{normal} if for each ordinal $0\leq\xi<\eta$ the flow $(X_{\xi+1},T)$ is the maximal isometric extension of $(X_{\xi},T)$ in $(X_{\eta},T)$.
This \emph{I}-system gives the minimal ordinal $\eta$ to represent the compact metric flow $(X,T)=(X_\eta,T)$ (cf. \cite{Fu}, Proposition 13.1, Definitions 13.2 and 13.3).
\end{defi}

It will be essential that the fibres of all the isometric extensions are connected, with a possible exception of extension from the trivial flow to the minimal isometric flow $(X_1,T)$.
For a \emph{normal} \emph{I}-system this property will be ensured by results of the paper \cite{MMWu}.
These results require the acting group to be generated by every open neighbourhood of a compact subset, and the group $T$ is even compactly generated.

\begin{proposition}
Let $\{(X_\xi,T):0\leq\xi\leq\eta\}$ be a normal \emph{I}-system with a compactly generated group $T$ acting minimally and distally.
Then for all ordinals $1\leq\xi<\zeta\leq\eta$ the extension from $(X_\xi,T)$ to $(X_\zeta,T)$ has connected fibres.
\end{proposition}

\begin{proof}
In the following arguments we shall refer to terminology and results provided in the paper \cite{MMWu}.
At first suppose that $1\leq\xi<\eta$ is not a limit ordinal.
Let $S(\pi_{\xi-1})\subset X^2$ denote the relativised equicontinuous structure relation of the homomorphism $\pi_{\xi-1}:(X,T)\longrightarrow(X_{\xi-1},T)$, hence the flow $(X,T)/S(\pi_{\xi-1})$ is the maximal isometric extension $(X_\xi,T)$ of the flow $(X_{\xi-1},T)$ in $(X,T)$.
By Theorem 3.7 of \cite{MMWu} the homomorphism $\pi_\xi:(X,T)\longrightarrow (X_\xi,T)$ has connected fibres.
For a limit ordinal $1<\xi<\eta$ and an ordinal $0\leq\zeta<\xi$, the same argument shows the connectedness of the fibres of $\pi_{\zeta+1}:(X,T)\longrightarrow(X_{\zeta+1},T)$.
Since
\begin{equation*}
(\pi_\xi)^{-1}(x_\xi)=\bigcap_{0\leq\zeta<\xi}(\pi_{\zeta+1})^{-1}(\pi_{\zeta+1}^\xi(x_\xi))
\end{equation*}
holds for every $x_\xi\in X_\xi$, the fibre $(\pi_\xi)^{-1}(x_\xi)$ is connected as the limit of a sequence of connected sets in a compact metric space (cf. \cite{K}, p. 170, Theorem 14).
The hypothesis follows, since for all ordinals $1\leq\xi<\zeta\leq\eta$ the fibres of the homomorphism $\pi_\xi^\zeta$ are the images under $\pi_\zeta$ of the connected fibres of $\pi_\xi$.
\end{proof}

We shall henceforth assume that $\{(X_\xi,T):0\leq\xi\leq\eta\}$ is a normal \emph{I}-system with $(X_\eta,T)=(X,T)$.
For every ordinal $1\leq\xi<\eta$ we shall define a projection of the cocycles of $(X,T)$ to the cocycles of $(X_\xi,T)$ by families of probability measures, using the fact that every distal extension of compact metric flows is a so-called \emph{RIM}-extension (relatively invariant measure, cf. \cite{Gl}).
For an isometric extension this \emph{RIM} is unique (cf. \cite{Gl}), and within the \emph{I}-system the \emph{RIM's} obey to an integral decomposition formula.

\begin{fact}
For every ordinal $0\leq\xi\leq\eta$ there exists a family of probability measures $\{\mu_{\xi,x_\xi}:x_\xi\in X_\xi\}$ on $X$ so that for every $x_\xi\in X_\xi$ and $\tau\in T$ holds
\begin{equation*}
\mu_{\xi,x_\xi}(\pi_\xi^{-1}(x_\xi))=1\enspace\textup{and}\enspace\mu_{\xi,x_\xi}\circ \tau^{-1}=\mu_{\xi, \tau x_\xi}.
\end{equation*}
The mapping $x_\xi\mapsto\mu_{\xi,x_\xi}$ is continuous with respect to the weak-* topology on $C(X)^*$.
For a continuous function $\varphi\in C(X)$ and ordinals $1\leq\xi<\zeta\leq\eta$ we have for all $x_\xi\in X_\xi$ the equality that
\begin{equation}
\mu_{\xi,x_\xi}(\varphi)=\int_{X_\zeta}\mu_{\zeta,y_\zeta}(\varphi)\, d(\mu_{\xi,x_\xi}\circ\pi_\zeta^{-1})(y_\zeta) .
\end{equation}
\end{fact}

\begin{proof}
The proof follows the inductive construction of an invariant measure for $(X,T)$ out of the unique \emph{RIM's} of the extensions $(X_\xi,T)=\pi_\xi^{\xi+1}(X_{\xi+1},T)$ with $0\leq\xi<\eta$ (cf. Chapter 12 of \cite{Fu}).
The generalisation to a relatively invariant measure is provided in \cite{dVr}, p. 494, where the induction process is initiated with the family of point measures on $X_\xi$ instead of the point measure on the trivial space $X_0$.
Since the \emph{I}-system and the \emph{RIM's} for the isometric extensions remain fixed, the decomposition formula follows from the inductive construction.
\end{proof}

Given a cocycle $f(\tau,x)$ of the flow $(X,T)$ and an ordinal $1\leq\xi<\eta$, the \emph{RIM} $\{\mu_{\xi,x_\xi}:x_\xi\in X_\xi\}$ defines a continuous function $f_\xi:T\times X_\xi\longrightarrow\R$ by
\begin{equation*}
f_\xi(\tau,x_\xi)=\mu_{\xi,x_\xi}(f(\tau,\cdot))=\int_X f(\tau,x)\,d\mu_{\xi,x_\xi}(x) .
\end{equation*}
The properties of the \emph{RIM} imply for all $\tau,\tau'\in T$ and $x_\xi\in X_\xi$ that
\begin{eqnarray*}
f_\xi(\tau,\tau'x_\xi)+f_\xi(\tau',x_\xi)=\mu_{\xi,\tau'x_\xi}(f(\tau,\cdot))+\mu_{\xi,x_\xi}(f(\tau',\cdot))=\\
=\mu_{\xi,x_\xi}(f(\tau,\cdot)\circ\tau')+\mu_{\xi,x_\xi}(f(\tau',\cdot))=f_\xi(\tau\tau',x_\xi),
\end{eqnarray*}
therefore $f_\xi$ is a cocycle of the flow $(X_\xi,T)$.
Furthermore, for ordinals $\xi$, $\zeta$ with $1\leq\xi<\zeta\leq\eta$ and every $x_\xi\in X_\xi$, the integral of the cocycle $(f_\zeta-f_\xi\circ\pi_\xi^\zeta)(\tau,x_\zeta)$ by the measure $\mu_{\xi,x_\xi}\circ\pi_\zeta^{-1}$ on $X_\zeta$ vanishes for every $\tau\in T$:
\begin{eqnarray*}
\int_{X_\xi}(f_\zeta-f_\xi\circ\pi_\xi^\zeta)(\tau,x_\zeta)\, d(\mu_{\xi,x_\xi}\circ\pi_\zeta^{-1})(x_\zeta) = \hspace{3cm}\\
= \int_{X_\xi}\left( \mu_{\zeta,x_\zeta}(f(\tau,\cdot))\right)\, d(\mu_{\xi,x_\xi}\circ\pi_\zeta^{-1})(x_\zeta) - \mu_{\xi,\pi_\xi^\zeta(x_\zeta)} (f(\tau,\cdot)) = 0
\end{eqnarray*}
Since the measure $\mu_{\xi,x_\xi}\circ\pi_\zeta^{-1}$ is supported by the connected fibre $(\pi_\xi^\zeta)^{-1}(x_\xi)$ in $X_\zeta$, for every $\tau\in T$ and every $x_\xi\in X_\xi$ the function $x_\zeta\mapsto (f_\zeta-f_\xi\circ\pi_\xi^\zeta)(\tau,x_\zeta)$ assumes zero on the fibre $(\pi_\xi^\zeta)^{-1}(x_\xi)$.
This property will be essential, as well as the representation of extensions of distal flows by so-called regular extensions.

\begin{fact}\label{fact:reg_ex}
Let $(X,T)$ be a distal minimal compact metric flow with a factor $(Y,T)=\sigma(X,T)$.
Then there exist a distal minimal compact Hausdorff flow $(\tilde X,T)$ with $(X,T)=\pi(\tilde X,T)$ as a factor and a Hausdorff topological group $G\subset\textup{Aut}(\tilde X,T)$ acting freely on $\tilde X$ (i.e. $g(\tilde x)=\tilde x$ for some $\tilde x\in\tilde X$ implies $g=\textbf 1_G$).
The group $G$ acts strictly transitive on the fibres $\tilde\sigma^{-1}(\tilde\sigma(\tilde x))=\{g(\tilde x):g\in G\}$ of the homomorphism $\tilde\sigma=\sigma\circ\pi$ for every $\tilde x\in\tilde X$, and $(X,T)$ is the orbit space of a subgroup $H$ of $G$ in $\tilde X$ so that $\pi$ is the mapping of a point in $\tilde X$ to its $H$-orbit (cf. \cite{El} 12.12, 12.13, and 14.26, with a direct proof in \cite{MMWu}, Proposition 1.1).

For an isometric extension $(X,T)$ of $(Y,T)$, the flow $(\tilde X,T)$ is metric and an isometric extension of $(Y,T)$, with a \emph{compact metric} group $G$ and a compact subgroup $H$, hence called a compact metric group extension.
\end{fact}

\begin{remark}\label{rem:d_n_m}
The construction above is also called the regulariser of an extension.
In \cite{Gl2} it is verified that a compact Hausdorff flow $(\tilde X,T)$ with these properties is metrisable if and only if the extension from $(Y,T)$ to $(X,T)$ is isometric.
\end{remark}

Studying the skew product extensions $\widetilde\tau_{f_\xi}:X_\xi\times\R\longrightarrow X_\xi\times\R$ for the ordinals $0\leq\xi\leq\eta$ will require the following technical lemma.

\begin{lemma}\label{lem:sub}
Let the minimal compact metric flow $(Z,T)$ be an extension of the flow $(Y,T)=\sigma(Z,T)$ and let $g(\tau,y)$ be a real-valued cocycle of $(Y,T)$.
Let $h(\tau,z)$ be a real-valued cocycle of $(Z,T)$ so that for every $\tau\in T$ and $z'\in Z$ the image of $\sigma^{-1}(\sigma(z'))$ under the function $z\mapsto h(\tau,z)$ is connected and includes zero.
Suppose that there exist a compact symmetric neighbourhood $K\subset\R$ of $0$ and $\varepsilon>0$ so that for all $z\in Z$ and $\tau\in T$ with $d_Z(z,\tau z)<\varepsilon$ holds $(g\circ\sigma+h)(\tau,z)\notin 2K\setminus K^0$.
Suppose that there exists a $\delta>0$ so that for all $\tau\in T$ and $z\in Z$ with $d_Z(z,\tau z)<\delta$ holds $d_Z(z',\tau z')<\varepsilon$ for every $z'\in\sigma^{-1}(\sigma(z))$.
Given $\bar z\in Z$ and a sequence $\{\bar\tau_k\}\subset T$ so that $\bar\tau_k \bar z$ converges and $(g\circ\sigma)(\bar\tau_k,\bar z)\to 0$ as $k\to\infty$, the sequence $\{(g\circ\sigma+h)(\bar\tau_k,\bar z)\}_{k\geq 1}$ is bounded.
Similarly, for a sequence $\{\bar\tau_k\}\subset T$ so that $\bar\tau_k\bar z$ converges and $(g\circ\sigma+h)(\bar\tau_k,\bar z)\to 0$, the sequence $\{(g\circ\sigma)(\bar\tau_k,\bar z)\}_{k\geq 1}$ is bounded.
\end{lemma}

\begin{proof}
There exists a $k_0\geq 1$ so that for all $k,k'\geq k_0$ holds $d_Z(\bar\tau_k\bar z,\bar\tau_{k'}\bar z)<\delta$ and
\begin{equation*}
(g\circ\sigma)(\bar\tau_{k'},\bar z)-(g\circ\sigma)(\bar\tau_k,\bar z)=(g\circ\sigma)(\bar\tau_{k'}\bar\tau_{k}^{-1},\bar\tau_k \bar z)\in K^0 .
\end{equation*}
By the choice of $K$, $\varepsilon$, and $\delta$ follows that $(g\circ\sigma+h)(\bar\tau_{k'}\bar\tau_k^{-1},z)\notin 2K\setminus K^0$ for all $z\in\sigma^{-1}(\sigma(\bar\tau_k\bar z))$.
Since the range of $(g\circ\sigma+h)(\bar\tau_{k'}\bar\tau_k^{-1},z)$ on the fibre $\sigma^{-1}(\sigma(\bar\tau_k\bar z))$ is connected and intersects $K^0$, we can conclude that $(g\circ\sigma+h)(\bar\tau_{k'}\bar\tau_k^{-1},\bar\tau_k \bar z)\in K^0$ for all $k,k'\geq k_0$.
Therefore the sequence $\{(g\circ\sigma+h)(\bar\tau_k,\bar z)\}_{k\geq 1}$ is bounded.

Provided a sequence $\{\bar\tau_k\}\subset T$ with convergent $\bar\tau_k\bar z$ and $(g\circ\sigma+h)(\bar\tau_k,\bar z)\to 0$, there exists an integer $k_0\geq 1$ so that for all $k,k'\geq k_0$ holds $d_Z(\bar\tau_k\bar z,\bar\tau_{k'}\bar z)<\delta$ and $(g\circ\sigma+h)(\bar\tau_{k'}\bar\tau_{k}^{-1},\bar\tau_k \bar z)\in K^0$.
We conclude as above that $(g\circ\sigma+h)(\bar\tau_{k'}\bar\tau_k^{-1},z)\in K^0$ for all $k,k'\geq k_0$ and $z\in\sigma^{-1}(\sigma(\bar\tau_k\bar z))$.
Since $h(\bar\tau_{k'}\bar\tau_k^{-1},z)=0$ for some $z\in\sigma^{-1}(\sigma(\bar\tau_k\bar z))$, the sequence $\{(g\circ\sigma)(\bar\tau_k,\bar z)\}_{k\geq 1}$ is bounded.
\end{proof}

At first the step from an ordinal to its successor by an isometric extension shall be considered.
The ``local'' behaviour within the fibres of a compact group extension is similar to a skew product extension by a compact metric group, even if the global structure might be different since it does not necessarily split into a product.

\begin{lemma}\label{lem:c_t}
Let $\gamma$ be an ordinal with $1\leq\gamma<\eta$.
If there exists a sequence $\{(\tau_k,x_k)\}_{k\geq 1}\subset T\times X_{\gamma+1}$ with $d_{\gamma+1} (x_k,\tau_k x_k)\to 0$ so that $(f_\gamma\circ\pi_\gamma^{\gamma+1})(\tau_k,x_k)\to 0$ and $f_{\gamma+1}(\tau_k,x_k)\nrightarrow 0$ as $k\to\infty$ (or equivalently $(f_{\gamma+1}\circ\pi_\gamma^{\gamma+1})(\tau_k,x_k)\nrightarrow 0$ and $f_\gamma(\tau_k,x_k)\to 0$), then the skew product $\widetilde\tau_{f_{\gamma+1}}$ is necessarily point transitive.
Therefore, if $f_\gamma(\tau,x_\gamma)$ is transient, then $f_{\gamma+1}(\tau,x_{\gamma+1})$ is either transient or the skew product $\widetilde\tau_{f_{\gamma+1}}$ is point transitive.
\end{lemma}

\begin{proof}
Suppose that $\widetilde\tau_{f_{\gamma+1}}$ is not point transitive and let $G\subset\textup{Aut}(Z,T)$ define a compact metric group extension of $(X_\gamma, T)$ with $(X_{\gamma+1},T)=\pi(Z,T)$.
Then the skew product extension $\widetilde\tau_{f_{\gamma+1}\circ\pi}$ of the flow $(Z,T)$ is also not point transitive, and Lemma \ref{lem:at} provides $K\subset\R$ and $\varepsilon>0$.
Since $G$ acts uniformly equicontinuous, there exists a $\delta>0$ so that for all $z\in Z$ and $\tau\in T$ with $d_Z (z,\tau z)<\delta$ follows $d_Z(k(z),k(\tau z))=d_Z(k(z),\tau k(z))<\varepsilon$ for all $k\in K$.
For every $z\in Z$ the $G$-orbit of $z$ is all of the fibre $(\pi_\gamma^{\gamma+1}\circ\pi)^{-1}((\pi_\gamma^{\gamma+1}\circ\pi)(z))$.
Since the $\pi^{\gamma+1}_\gamma$-fibres are connected, for every $\tau\in T$ and $z'\in Z$ the range of $(f_{\gamma+1}-f_\gamma\circ\pi^{\gamma+1}_\gamma)(\tau,\pi(z))$ on the fibre $(\pi_\gamma^{\gamma+1}\circ\pi)^{-1}((\pi_\gamma^{\gamma+1}\circ\pi)(z'))$ is connected and contains zero.
Hence Lemma \ref{lem:sub} applies with $(Y,T)=(X_\gamma,T)$, $\sigma=\pi_\gamma^{\gamma+1}\circ\pi$, $g=f_\gamma$, and $h(\tau,z)=(f_{\gamma+1}-f_\gamma\circ\pi^{\gamma+1}_\gamma)(\tau,\pi(z))$.

However, given the sequence $\{(\tau_k,x_k)\}_{k\geq 1}\subset T\times X_{\gamma+1}$ in the hypothesis, Lemma \ref{lem:trans} provides a point $\bar x\in X_{\gamma+1}$ and a sequence $\{\bar\tau_k\}\subset T$ so that $(f_\gamma\circ\pi_\gamma^{\gamma+1})(\bar\tau_k,\bar x)\to 0$ and $f_{\gamma+1}(\bar\tau_k,\bar x)\to\infty$ (or $(f_\gamma\circ\pi_\gamma^{\gamma+1})(\bar\tau_k,\bar x)\to\infty$ and $f_{\gamma+1}(\bar\tau_k,\bar x)\to 0$).
By choosing a point $\bar z\in Z$ with $\pi(\bar z)=\bar x$ and changing to a subsequence of $\{\bar\tau_k\}\subset T$ with $\bar\tau_k\bar z$ convergent, this contradicts to Lemma \ref{lem:sub}.

Now suppose that $f_\gamma(\tau,x_\gamma)$ is transient and $f_{\gamma+1}(\tau,x_{\gamma+1})$ is recurrent.
Let $x'\in X_{\gamma+1}$ be so that $(x',0)$ is $\widetilde\tau_{f_{\gamma+1}}$-recurrent (cf. Remarks \ref{rems:rec}).
Since $(x',0)$ is cannot be $\widetilde\tau_{f_\gamma\circ\pi_\gamma^{\gamma+1}}$-recurrent, there exist a neighbourhood $V\subset X_{\gamma+1}\times\R$ of $(x',0)$ and a replete semigroup $P\subset T$ so that $\widetilde\tau_{f_\gamma\circ\pi_\gamma^{\gamma+1}}(x',0)\notin V$ for every $\tau\in P$.
Given an arbitrary compact set $C\subset T$, by Theorem 6.32 in \cite{G-H} there exists a replete semigroup $Q\subset P\setminus C$.
Since $(x',0)$ is $\widetilde\tau_{f_{\gamma+1}}$-recurrent, we can inductively construct a sequence $\{\tau_k\}_{k\geq 1}\subset P$ with $\tau_k x'\to x'$, $f_{\gamma+1}(\tau_k,x')\to0$, and $(f_\gamma\circ\pi^{\gamma+1}_\gamma)(\tau_k,x')\nrightarrow 0$.
The point transitivity of $\widetilde\tau_{f_{\gamma+1}}$ follows by the preceding statement.
\end{proof}

Furthermore, we shall study the case of transfinite induction to a limit ordinal.
The arguments are quite similar, however with an approximation of a limit ordinal instead of an isometric group extension.

\begin{lemma}\label{lem:lim}
Suppose that $\gamma$ is a limit ordinal with $1<\gamma\leq\eta$.
\begin{enumerate}
\item If for every ordinal $1<\alpha<\gamma$ there exists an ordinal $\alpha\leq\xi<\gamma$ so that $f_\xi(\tau,x_\xi)$ has a point transitive skew product extension, then $f_\gamma(\tau,x_\gamma)$ has a point transitive skew product extension.
\item If there exist an ordinal $1\leq\alpha<\gamma$ and a sequence $\{(\tau_k,x_k)\}_{k\geq 1}\subset T\times X_\gamma$ with $d_\gamma (x_k,\tau_k x_k)\to 0$  so that $(f_\xi\circ\pi_\xi^\gamma)(\tau_k,x_k)\to 0$ for every $\alpha\leq\xi<\gamma$ and $f_\gamma(\tau_k,x_k)\nrightarrow 0$ as $k\to\infty$ (or equivalently $(f_\xi\circ\pi_\xi^\gamma)(\tau_k,x_k)\nrightarrow 0$ for every $\alpha\leq\xi<\gamma$ and $f_\gamma(\tau_k,x_k)\to 0$), then $\widetilde\tau_{f_\gamma}$ is necessarily point transitive.
\item If there exists an ordinal $1\leq\alpha<\gamma$ so that for all $\alpha\leq\xi<\gamma$ the cocycle $f_\xi(\tau,x_\xi)$ is transient, then $f_\gamma(\tau,x_\gamma)$ is either transient or its skew product extension is point transitive.
\end{enumerate}
\end{lemma}

\begin{proof}
Suppose that the skew product of $\tau_{f_\gamma}$ on $X_\gamma\times\R$ is not point transitive, and let $K\subset\R$ and $\varepsilon>0$ be provided by Lemma \ref{lem:at}.
Since $\gamma$ is a limit ordinal and $(X_\gamma,T)$ is the inverse limit of the flows $\{(X_\xi,T):0\leq\xi<\gamma\}$, we can choose an ordinal $\zeta<\gamma$ so that for all $x,x'\in X_\gamma$ with $\pi_\zeta^\gamma (x)=\pi_\zeta^\gamma (x')$ holds $d_\gamma (x,x')<\varepsilon/3$.
If we put $\delta=\varepsilon/3$, then $d_\gamma(x',\tau x')<\delta$ for $x'\in X_\gamma$ and $\tau\in T$ implies $d_\gamma (x,\tau x)<\varepsilon$ for all $x\in(\pi^\gamma_\zeta)^{-1}(\pi^\gamma_\zeta(x'))$.
These conditions remain valid even if the ordinal $\zeta$ will be increased later.
Since the $\pi_\zeta^{\gamma}$-fibres are connected, for every $\tau\in T$ and $x'_\gamma\in X_\gamma$ the range of $(f_{\gamma}-f_\zeta\circ\pi^{\gamma}_\zeta)(\tau,x_\gamma)$ on the $\pi_\zeta^{\gamma}$-fibre of $x'_\gamma$ is connected and contains $0$.

Under the hypothesis (i), we can choose $(\tau,x_\zeta)\in T\times X_\zeta$ so that $f_\zeta(\tau,x_\zeta)\in 2K\setminus K^0$ and $d_\gamma(x'_\gamma,\tau x'_\gamma)<\delta$ for some $x_\gamma'\in(\pi_\zeta^{\gamma})^{-1}(x_\zeta)$.
Thus $d_\gamma(x_\gamma,\tau x_\gamma)<\varepsilon$ holds for all $x_\gamma\in(\pi_\zeta^{\gamma})^{-1}(x_\zeta)$, and for $x_\gamma$ with $(f_{\gamma}-f_\zeta\circ\pi^{\gamma}_\zeta)(\tau,x_\gamma)=0$ this contradicts $f_\gamma(\tau,x_\gamma)\notin 2K\setminus K^0$.
Thus assertion (i) is verified.

We apply then Lemma \ref{lem:sub} with $(Z,T)=(X_\gamma,T)$, $(Y,T)=(X_\zeta,T)$, $\sigma=\pi_\zeta^{\gamma}$, $h=(f_{\gamma}-f_\zeta\circ\pi^{\gamma}_\zeta)$, and $g=f_\zeta$.
However, given the sequence $\{(\tau_k,x_k)\}_{k\geq 1}\subset T\times X_\gamma$ in hypothesis (ii), Lemma \ref{lem:trans} provides a point $\bar x\in X_\gamma$ and a sequence $\{\bar\tau_k\}\subset T$ so that $(f_\zeta\circ\pi_\zeta^{\gamma},f_{\gamma}-f_\zeta\circ\pi^{\gamma}_\zeta)(\bar\tau_k,\bar x)=(g\circ\sigma,h)(\bar\tau_k,\bar x)\to(0,\infty)$ (or $(\infty,0)$) and $\bar\tau_k\bar x\to\bar z$ as $k\to\infty$.
This is a contradiction to Lemma \ref{lem:sub} and verifies (ii).

Now suppose that $f_\zeta(\tau,x_\zeta)$ is transient and $f_{\gamma}(\tau,x_{\gamma})$ is recurrent, and choose $x'\in X_\gamma$ so that $(x',0)$ is $\widetilde\tau_{f_{\gamma}}$-recurrent.
Since $(x',0)$ is cannot be $\widetilde\tau_{f_\zeta\circ\pi_\zeta^{\gamma}}$-recurrent, there exist a neighbourhood $V\subset X_\gamma\times\R$ of $(x',0)$ and a replete semigroup $P\subset T$ with $\widetilde\tau_{f_\zeta\circ\pi_\zeta^{\gamma}}(x',0)\notin V$ for every $\tau\in P$.
By induction exists a sequence $\{\tau_k\}_{k\geq 1}\subset P$ with $\widetilde{(\tau_k)}_{f_{\gamma}}(x',0)\to(x',0)$ as $k\to\infty$, and by Lemma \ref{lem:trans} there exist a point $\bar x\in X_\gamma$ and sequence $\{\bar\tau_k\}\subset T$ so that $(f_\gamma,f_\zeta\circ\pi^{\gamma}_\zeta)(\bar\tau_k,\bar x)=(g\circ\sigma+h,g\circ\sigma)(\bar\tau_k,\bar x)\to(0,\infty)$ and $\bar\tau_k\bar x\to\bar x$.
This contradiction to Lemma \ref{lem:sub} verifies the statement (iii).
\end{proof}

\begin{proposition}\label{prop:max}
If the real-valued cocycle $f(\tau,x)$ is topologically recurrent apart from a coboundary, then there exists a maximal ordinal $1\leq\alpha\leq\eta$ so that the skew product extension $\widetilde\tau_{f_\alpha}$ is point transitive on $X_\alpha\times\R$.
The cocycle $(f-f_\alpha\circ\pi_\alpha) (\tau,x)$ is relatively trivial with respect to $(f_\alpha\circ\pi_\alpha)(\tau,x)$.
\end{proposition}

\begin{proof}
Let us first suppose that the cocycle $f_\xi(\tau,x_\xi)$ is recurrent for every ordinal $1\leq\xi<\eta$, and let $\mc O=\{1\leq\xi\leq\eta:{f_\xi}(\tau,x_\xi)\enspace\textup{is \emph{not} a coboundary}\}$.
This set is non-empty since $f_\eta(\tau,x)$ is not a coboundary, and let $\beta$ be its minimal element.
If $\beta=1$, then by Proposition \ref{prop:isom} the recurrent skew product $\widetilde{\tau}_{f_1}$ of the isometric flow $(X_1,T)$ is point transitive.
If $\beta>1$, then Fact \ref{fact:GH} provides a sequence $\{(\tau_k,x_k)\}_{k\geq 1}\subset T\times X_\beta$ with $d_\beta(x_k,\tau_k x_k)\to 0$ and $f_\beta(\tau_k, x_k)\to\infty$.
For all $1\leq\zeta<\beta$ holds $(f_\beta\circ\pi_\zeta^\beta)(\tau_k, x_k)\to 0$, and by the Lemmas \ref{lem:c_t} and \ref{lem:lim} (ii) $\widetilde{\tau}_{f_\beta}$ is point transitive.

If $f_\xi(\tau,x_\xi)$ is transient for an ordinal $1\leq\xi<\eta$, then let $\beta$ be the minimal element of the set $\mc O=\{\xi<\zeta\leq\eta:f_\zeta(\tau,x_\zeta)\enspace\textup{is topologically recurrent}\}$.
This set is non-empty since $f_\eta(\tau,x_\eta)$ is topologically recurrent, and it follows from the Lemmas \ref{lem:c_t} and \ref{lem:lim} (iii) that $\widetilde\tau_{f_\beta}$ is even point transitive.

Now let $\mc O=\{1\leq\xi\leq\eta:\widetilde\tau_{f_\zeta}\enspace\textup{is \emph{not} point transitive for all}\enspace\xi\leq\zeta\leq\eta\}$.
If $\mc O$ is empty, then $\widetilde\tau_{f_\eta}$ is point transitive and $\alpha=\eta$.
Otherwise, the set $\mc O$ has a minimal element $\gamma>1$ since $\widetilde\tau_{f_\beta}$ is point transitive for some $1\leq\beta\leq\eta$.
Since $\gamma$ cannot be a limit ordinal by Lemma \ref{lem:lim} (i), there exists a maximal ordinal $\alpha\geq 1$ with point transitive $\widetilde\tau_{f_\alpha}$.
If $\{(\tau_k,x_k)\}_{k\geq 1}\subset T\times X$ is a sequence with $d (x_k,\tau_k x_k)\to 0$ and $(f_\alpha\circ\pi_\alpha)(\tau_k,x_k)\to 0$, then transfinite induction using the maximality of $\alpha$ and Lemmas \ref{lem:c_t}, \ref{lem:lim} (ii) verifies that $(f_\xi\circ\pi_\xi)(\tau_k,x_k)\to 0$ for every $\alpha\leq\xi\leq\eta$.
\end{proof}

After the flow $(X_\alpha,T)$ with a point transitive skew product extension $\widetilde\tau_{f_\alpha}$ has been identified, we shall study the extension from $(X_\alpha,T)$ to $(X,T)$.
There might be infinitely many isometric extensions in between, and therefore this extension is in general a distal extension.
Since our construction will use the regulariser of this extension, it is necessary to leave the category of compact metric flows for the category of compact Hausdorff flows during the following construction (cf. Remark \ref{rem:d_n_m}).
However, the flow which will be constructed by means of the regulariser will be metric as a factor of the compact metric flow $(X,T)$.

\begin{proposition}\label{prop:flow}
There exists a factor $(Y,T)=(X_\alpha\times M,T)=\pi_Y(X,T)$ which is a Rokhlin extension of $(X_\alpha,T)=\rho_\alpha(Y,T)$ by a distal minimal compact metric $\R$-flow $(M,\{\phi^t:t\in\R\})$ and the cocycle $f_\alpha(\tau,x_\alpha)$ so that for every $x\in X$ holds
\begin{equation}\label{eq:p_X}
\pi_Y^{-1}(\pi_Y(x))\times\{0\}=\mc D_{T, f_\alpha\circ\pi_\alpha}(x,0)\cap(\pi_\alpha^{-1}(\pi_\alpha(x))\times\{0\}) .
\end{equation}
The $\R$-flow $\{\psi^t:t\in\R\}\subset\textup{Aut}(Y,T)$ defined by $\psi^t(x_\alpha,m)=(x_\alpha,\phi^t(m))$ for $(x_\alpha,m)\in Y=X_\alpha\times M$ fulfils for every $y\in Y$ and every $t\in\R$ that
\begin{equation}\label{eq:o_flow}
\bar{\mc O}_{T, f_\alpha\circ\rho_\alpha}(y,0)\cap(\rho_\alpha^{-1}(\rho_\alpha(y))\times\{t\})\subset\{(\psi^t(y),t)\} ,
\end{equation}
with coincidence of these sets if $(\rho_\alpha(y),0)\in X_\alpha\times\R$ is transitive for $\widetilde\tau_{f_\alpha}$.
\end{proposition}

\begin{proof}
We shall construct a factor $(Y,T)$ of $(X,T)$ and a flow $\{\varphi^t:t\in\R\}\subset\textup{Aut}(Y,T)$, and then we shall represent $(Y,T)$ as a Rokhlin extension of $(X_\alpha,T)$.
Let $(\tilde X,T)$ be a distal minimal compact Hausdorff flow with $(X,T)=\pi(\tilde X,T)$ and a Hausdorff topological group $G\subset\textup{Aut}(\tilde X,T)$ acting freely on the fibres of $\pi_\alpha\circ\pi$ so that $(X,T)$ is the $H$-orbit space of a subgroup $H\subset G$ (cf. Fact \ref{fact:reg_ex}).
For an arbitrary point $\tilde z\in \tilde X$ and $t\in\R$ we define a closed subset of $G$ by
\begin{equation}\label{eq:G}
G_{\tilde z,t}=\{g\in G: (\pi(g(\tilde z)),t)\in\mc D_{T,f_\alpha\circ\pi_\alpha}(\pi(\tilde z),0)\} .
\end{equation}
The mapping $\pi$ is open as a homomorphism of distal minimal compact flows, and hence for every $g\in G_{\tilde z,t}$ there exist nets $\{\tilde z_i\}_{i\in I}\subset \tilde X$ and $\{\tau_i\}_{i\in I}\subset T$ with $\tilde z_i\to \tilde z$, $\tau_i\pi(\tilde z_i)\to\pi(g(\tilde z))$, and $f_\alpha(\tau_i,\pi_\alpha\circ\pi(\tilde z_i))\to t$.
Since the cocycle $(f_\alpha\circ\pi_\alpha)(\tau,x_\alpha)$ is constant on the fibres of $\pi_\alpha$ and $T$ is Abelian, it follows for every fixed $\tau\in T$ that
\begin{equation*}
\tau_i\pi(\tau \tilde z_i)=\tau_i\tau\pi(\tilde z_i)=\tau\tau_i\pi(\tilde z_i)\to\tau\pi(g(\tilde z))=\pi(\tau g(\tilde z))=\pi(g(\tau \tilde z))
\end{equation*}
and by the cocycle identity
\begin{eqnarray*}
f_\alpha(\tau_i,\pi_\alpha\circ\pi(\tau \tilde z_i)) & = & f_\alpha(\tau_i,\pi_\alpha\circ\pi(\tilde z_i))-f_\alpha(\tau,\pi_\alpha\circ\pi(\tilde z_i))\\
& & +f_\alpha(\tau,\pi_\alpha\circ\pi(\tau_i\tilde z_i))\to t .
\end{eqnarray*}
By the density of the $T$-orbit of $\tilde z$ and a diagonalisation of nets there exist for every $\tilde x\in \tilde X$ nets $\{\tilde x_i\}_{i\in I}\subset \tilde X$ and $\{\tau'_i\}_{i\in I}\subset T$ with $\tilde x_i\to \tilde x$, $\tau'_i\pi(\tilde x_i)\to\pi(g(\tilde x))$, and $f_\alpha(\tau_i,\pi_\alpha\circ\pi(\tilde x_i))\to t$.
Therefore
\begin{equation*}
(\pi(g(\tilde x)),t)\in\mc D_{T,f_\alpha\circ\pi_\alpha}(\pi(\tilde x),0)
\end{equation*}
so that $g\in G_{\tilde x,t}=G_{\tilde z,t}=G_t$.
By symmetry follows now that $G_{-t}=(G_t)^{-1}$.

Then we fix a point $x'\in X$ with $\bar{\mc O}_{T,f_\alpha}(\pi_\alpha(x'),0)=X_\alpha\times\R$ and $\mc D_{T,f_\alpha\circ\pi_\alpha}(x',0)=\bar{\mc O}_{T,f_\alpha\circ\pi_\alpha}(x',0)$ (cf. Fact \ref{fact:o_p}).
The set $G_t$ is non-empty for every $t\in\R$, since $\bar{\mc O}_{T, f_\alpha}(\pi_\alpha(x'),0)=X_\alpha\times\R$ and the compactness of $X$ ensure that
\begin{equation*}
\bar{\mc O}_{T,f_\alpha\circ\pi_\alpha}(x',0)\cap\pi_\alpha^{-1}(\pi_\alpha(x'))\times\{t\}\neq\emptyset .
\end{equation*}
For arbitrary $t,t'\in\R$ and $g\in G_t$, $g'\in G_{t'}$, we select $\tilde x, \tilde z\in \tilde X$ so that $\pi(\tilde x)=x'$ and $\tilde x=g'(\tilde z)$.
Then we have
\begin{equation*}
(x',t')=(\pi(g'(\tilde z)),t')\in\mc D_{T,f_\alpha\circ\pi_\alpha}(\pi(\tilde z),0) ,
\end{equation*}
and for $\tilde y=g(\tilde x)=gg'(\tilde z)$ it holds that $(\pi(\tilde y),t)\in\bar{\mc O}_{T,f_\alpha\circ\pi_\alpha}(x',0)=\mc D_{T,f_\alpha\circ\pi_\alpha}(x',0)$.
By Remark \ref{rem:o_p} follows $(\pi(\tilde y),t+t')\in\mc D_{T,f_\alpha\circ\pi_\alpha}(\pi(\tilde z),0)$ so that $gg'\in G_{t+t'}$.
Hence $G_t G_{t'}\subset G_{t+t'}$ holds for all $t,t'\in\R$, and from $G_{-t}=(G_{t})^{-1}$ follows $(G_{t})^{-1} G_{t+t'}=G_{-t} G_{t+t'}\subset G_{t'}$ so that $G_{t} G_{t'}= G_{t+t'}$.
Thus the Hausdorff topological group
\begin{equation*}
\tilde G=\cup_{t\in\R}G_t
\end{equation*}
has the closed set $G_0\supset H$ as a normal subgroup so that $G_t$ is a $G_0$-coset in $\tilde G$ for every $t\in\R$.
Moreover, the mapping $t\mapsto G_t$ is a group homomorphism from $\R$ into $\tilde G/G_0$.
The group $G_0$ is not necessarily compact, however its orbit space on $\tilde X$ defines a partition into sets invariant under $H\subset G_0$.
Hence this is also a partition of $X$, and the equivalence relation $R_Y$ of this partition of $X$ is $T$-invariant since $G_0\subset\textup{Aut}(\tilde X,T)$.
Moreover, $R_Y$ is closed in $X^2$, since definition (\ref{eq:G}) implies that $(x,x')\in R_Y$ if and only if $(x',0)\in\mc D_{T, f_\alpha\circ\pi_\alpha}(x,0)\cap(\pi_\alpha^{-1}(\pi_\alpha(x))\times\{0\})$.
The factor $(Y,T)=\pi_Y(X,T)$ defined by the $T$-invariant closed equivalence relation $R_Y$ is an extension of $(X_\alpha,T)=\rho_\alpha(Y,T)$, and equality (\ref{eq:p_X}) follows.
The $\R$-action $\{\varphi^t:t\in\R\}\subset\textup{Aut}(Y,T)$ is well defined for every $y\in Y$ and $t\in\R$ by
\begin{equation*}
\varphi^t(y)=G_t ((\pi_Y\circ\pi)^{-1}(y))=G_t (\{\tilde x\in \tilde X:G_0(\tilde x)=y\}) .
\end{equation*}
Let $\{(t_k,y_k)\}_{k\geq 1}\subset\R\times Y$ be a sequence with $(t_k,y_k)\to (t,y)$, then $\varphi^{t_k}(y_k)=G_0 g_k(\tilde x_k)$ for a sequence $\{\tilde x_k\}_{k\geq 1}\subset\tilde X$ with $\pi_Y\circ\pi(\tilde x_k)=y_k$ and $g_k\in G_{t_k}$.
We can assume that $\tilde x_k\to\tilde x$ and $g_k(\tilde x_k)\to\tilde z$ so that $(\pi(\tilde z),t)\in\mc D_{T,f_\alpha\circ\pi_\alpha}(\pi(\tilde x),0)$ and $\tilde z=g_t(\tilde x)$ for some $g_t\in G_t$.
From $\pi_Y\circ\pi(\tilde x)=y$ and $\varphi^{t_k}(y_k)=\pi_Y\circ\pi(g_k(\tilde x_k))\to\pi_Y\circ\pi(\tilde z)=\varphi^t(y)$ follows the continuity of the action $\{\varphi^t:t\in\R\}$ on $Y$.

We turn to the inclusion (\ref{eq:o_flow}).
Suppose that $(y_i,t)\in\bar{\mc O}_{T, f_\alpha\circ\rho_\alpha}(y,0)\cap\rho_\alpha^{-1}(x_\alpha)\times\{t\}$ for some $x_\alpha\in X_\alpha$ and $i\in\{1,2\}$, and select $x\in\pi_Y^{-1}(y)$.
By the compactness of $X$ there exist points $x_i\in\pi_Y^{-1}(y_i)\subset\pi_\alpha^{-1}(x_\alpha)$ so that $(x_i,t)\in\bar{\mc O}_{T, f_\alpha\circ\pi_\alpha}(x,0)$, and therefore $(x_2,0)\in\mc D_{T,f_\alpha\circ\pi_\alpha}(x_1,0)$.
The equality (\ref{eq:p_X}) implies that $y_1=\pi_Y(x_1)=\pi_Y(x_2)=y_2$, and thus for every $y\in Y$ and $t\in\R$ holds
\begin{equation}\label{eq:card_1}
\textup{card}\{\bar{\mc O}_{T, f_\alpha\circ\rho_\alpha}(y,0)\cap\rho_\alpha^{-1}(\rho_\alpha(y))\times\{t\}\}\leq 1 .
\end{equation}
Moreover, for $x_\alpha=\rho_\alpha(y)$ follows $x_1=\pi(g_t(\tilde x))$ with $g_t\in G_t$ and $\tilde x\in\pi^{-1}(x)\subset\tilde X$.
Hence $y_1=\pi_Y(x_1)=\varphi^t(y)$ and inclusion (\ref{eq:o_flow}) is verified.
For $\widetilde\tau_{f_\alpha}$-transitive $(\rho_\alpha(y),0)$ the cardinality in (\ref{eq:card_1}) is equal to $1$ for every $y\in Y$ and $t\in\R$, and for $y'\in\rho_\alpha^{-1}(\rho_\alpha(y))$ and $x_\alpha\in X_\alpha$ holds $\bar{\mc O}_{T,f_\alpha\circ\rho_\alpha}(y',0)\cap\rho_\alpha^{-1}(x_\alpha)\times\{0\}=\{(y_1,0)\}$.
We fix a point $\bar x\in X_\alpha$ with $\widetilde\tau_{f_\alpha}$-transitive $(\bar x,0)$.
If $(y_2,0)\in\mc D_{T,f_\alpha\circ\rho_\alpha}(y',0)\cap\rho_\alpha^{-1}(x_\alpha)\times\{0\}$, then Remark \ref{rem:o_p} implies that $(y_2,0)\in\mc D_{T,f_\alpha\circ\rho_\alpha}(y_1,0)$, and as above follows $y_1=y_2$.
Hence
\begin{equation}\label{eq:homeo_Y}
\bar{\mc O}_{T,f_\alpha\circ\rho_\alpha}(y',0)\cap\rho_\alpha^{-1}(x_\alpha)\times\{0\}=\mc D_{T,f_\alpha\circ\rho_\alpha}(y',0)\cap\rho_\alpha^{-1}(x_\alpha)\times\{0\}
\end{equation}
holds for every $y'\in\rho_\alpha^{-1}(\bar x)$ and $x_\alpha\in X_\alpha$.
For distinct $y',y''\in\rho_\alpha^{-1}(\bar x)$ follows
\begin{equation*}
\bar{\mc O}_{T,f_\alpha\circ\rho_\alpha}(y',0)\cap\bar{\mc O}_{T,f_\alpha\circ\rho_\alpha}(y'',0)\cap\rho_\alpha^{-1}(x_\alpha)\times\{0\}=\emptyset .
\end{equation*}
Indeed, given a point $\bar y$ in this intersection, for every sequence $\{\tau_k\}_{k\geq 1}\subset T$ with $\tau_k\bar x\to\rho_\alpha(\bar y)$ and $f_\alpha(\tau_k,\bar x)\to 0$ follows by equality (\ref{eq:card_1}) that $d_Y(\tau_k y',\tau_k y'')\to 0$, in contradiction to the distality of $(Y,T)$.
Hence the mapping $\iota:X_\alpha\times\rho_\alpha^{-1}(\bar x)\longrightarrow Y$
\begin{equation*}
(x_\alpha,y')\mapsto\rho_\alpha^{-1}(x_\alpha)\cap\{y\in Y:(y,0)\in\bar{\mc O}_{T,f_\alpha\circ\rho_\alpha}(y',0)\}
\end{equation*}
is well-defined, one-to-one, and by equality (\ref{eq:homeo_Y}) also continuous.
For a dense set of points $\bar y\in Y$ holds the $\widetilde\tau_{f_\alpha}$-transitivity of $(\rho_\alpha(\bar y),0)$, since $\rho_\alpha$ is open.
We can conclude for every $y\in Y$ that $\mc D_{T,f_\alpha\circ\rho_\alpha}(y,0)\cap\rho_\alpha^{-1}(\bar x)\times\{0\}\neq\emptyset$, and thus $\mc D_{T,f_\alpha\circ\rho_\alpha}(y',0)\cap\{y\}\times\{0\}=\bar{\mc O}_{T,f_\alpha\circ\rho_\alpha}(y',0)\cap\{y\}\times\{0\}\neq\emptyset$ for some $y'\in\rho_\alpha^{-1}(\bar x)$.
Hence $\iota$ is onto and by compactness $Y$ and $X_\alpha\times\rho_\alpha^{-1}(\bar x)$ are homeomorphic.

Let $\{\phi^t:t\in\R\}$ be the restriction of $\{\varphi^t:t\in\R\}$ to the $\{\varphi^t:t\in\R\}$-invariant compact metric space $M=\rho_\alpha^{-1}(\bar x)$.
For every $y'\in M$ and $\tau\in T$ holds $\bar{\mc O}_{T, f_\alpha\circ\rho_\alpha}(y',0)\cap\rho_\alpha^{-1}(\bar x)\times\{-f_\alpha(\tau,\bar x) \}=\{(\phi^{-f_\alpha(\tau,\bar x)}(y'),-f_\alpha(\tau,\bar x))\}$ and $\widetilde{\tau}_{f_\alpha\circ\pi_Y}(\phi^{-f_\alpha(\tau,\bar x)}(y'),-f_\alpha(\tau,\bar x))\in\rho_\alpha^{-1}(\tau\bar x)\cap\{y\in Y:(y,0))\in\bar{\mc O}_{T,f_\alpha\circ\rho_\alpha}(y',0)\}$.
Therefore $\tau\phi^{-f_\alpha(\tau,\bar x)}(y')=\iota(\tau\bar x,y')$ and $\tau y=\iota(\tau_{\phi,f_\alpha}(\bar x,y))$ for every $y\in M$ and $\tau\in T$.
The minimality of $(Y,T)$ implies that $(X_\alpha\times M,T)$ and $(Y,T)$ are topologically isomorphic via $\iota$.
Moreover, for the mapping $\psi^t(x_\alpha,m)=(x_\alpha,\phi^t(m))$ with $\psi\in\textup{Aut}(X_\alpha\times M,T)$ and every $m\in M=\rho_\alpha^{-1}(\bar x)$ and $t\in\R$ holds $\varphi^t(\iota(\tau_{\phi,f_\alpha}(\bar x,m)))=\varphi^t(\tau m)=\tau\varphi^t(m)=\iota(\tau_{\phi,f_\alpha}(\bar x,\phi^t(m)))=\iota(\psi^t(\tau_{\phi,f_\alpha}(\bar x,m)))$.
By the minimality of $(X_\alpha\times M,T)$ follows $\psi^t=\iota^{-1}\circ\varphi^t\circ\iota$ for every $t\in\R$.
The flow $(M,\{\phi^t:t\in\R\})$ is minimal and distal, since a non-transitive point $m'\in M$ and a proximal pair $(m',m'')\in M^2$, respectively, give rise a non-transitive point $(x_\alpha,m')\in Y$ and a proximal pair $((x_\alpha,m'),(x_\alpha,m''))\in Y^2$, respectively.
\end{proof}

It should be mentioned that an ordinal $\xi\leq\eta$ with $(Y,T)=(X_\xi,T)$ does not necessarily exist.
Therefore we shall define a cocycle $f_Y:T\times Y\longrightarrow\R$ independently of the cocycles $f_\xi(\tau,x_\xi)$, and it will turn out that $(f_Y\circ\pi_Y)(\tau,x)$ can be chosen topologically cohomologous to $f$.

\begin{proposition}\label{prop:res}
There exists a topological cocycle $f_Y(\tau,y)$ of the flow $(Y,T)$ so that $(f_Y\circ\pi_Y)(\tau,x)$ is topologically cohomologous to $f(\tau,x)$ and $f_Y(\tau,y)$ is relatively trivial with respect to $(f_\alpha\circ\rho_\alpha)(\tau,y)$.
\end{proposition}

We shall prove another technical lemma first.

\begin{lemma}\label{lem:tilde_coc}
Let $(Z,T)$ be a distal minimal compact metric flow which extends $(X_\alpha,T)=\sigma_\alpha(Z,T)$, and let $G\subset\textup{Aut}(Z,T)$ be a Hausdorff topological group preserving the fibres of $\sigma_\alpha$.
Suppose that there exists a continuous group homomorphism $\varphi:G\longrightarrow\R$ so that for every $g\in G$ and every $z\in Z$ it holds that
\begin{equation*}
(g(z),\varphi(g))\in\mc D_{T,f_\alpha\circ\sigma_\alpha}(z,0) .
\end{equation*}
Furthermore, suppose that $h(\tau,z)$ is a real-valued cocycle of $(Z,T)$ which is relatively trivial with respect to $(f_\alpha\circ\sigma_\alpha)(\tau,z)$.
Then there exists a continuous cocycle $\bar h((\tau,g),z)$ of the flow $(Z,T\times G)$ with the action $\{g\circ\tau:(\tau,g)\in T\times G\}$ so that $h(\tau,z)=\bar h((\tau,\mathbf 1_G),z)$ holds for every $(\tau,z)\in T\times Z$ and the mapping $h\mapsto\bar h$ is linear.
For $z\in Z$, $g\in G$, and a sequence $\{(\tau_k,z_k)\}_{k\geq 1}\subset T\times Z$ with $z_k\to z$, $\tau_k z_k\to g(z)$, and $(f_\alpha\circ\rho_\alpha)(\tau_k,z_k)\to\varphi(g)$ holds
\begin{equation}\label{eq:h_uni}
h(\tau_k,z_k)\to\bar h((\mathbf 1_T,g),z)\quad\textup{as}\enspace k\to\infty.
\end{equation}
\end{lemma}

\begin{proof}
We put $F=(f_\alpha\circ\sigma_\alpha,h):T\times Z\longrightarrow\R^2$ and fix a point $\bar z\in Z$ so that $\bar{\mc O}_{T,F} (\bar z,0,0)$ and $\mc D_{T,F} (\bar z,0,0)$ coincide in $Z\times(\R_\infty)^2$ (cf. Fact \ref{fact:o_p}).
For every $g\in G$ we fix a sequence $\{\tau_k^g\}_{k\geq 1}\subset T$ with $\widetilde{(\tau_k^g)}_{f_\alpha\circ\sigma_\alpha}(\bar z,0)\to(g(\bar z),\varphi(g))$ as $k\to\infty$, with $\{\tau_k^{\mathbf 1_G}=\mathbf 1_T\}_{k\geq 1}$.
Since $g\in\textup{Aut}(Z,T)$ and $f_\alpha\circ\sigma_\alpha\circ g=f_\alpha\circ\sigma_\alpha$, we can conclude for every $t\in T$ that $\tau_k^gt\bar z=t\tau_k^g\bar z\to t g(\bar z)=g(t\bar z)$ as $k\to\infty$ as well as
\begin{equation}\label{eq:transf}
(f_\alpha\circ\sigma_\alpha)(\tau_k^g,t\bar z)=(f_\alpha\circ\sigma_\alpha)(t, \tau_k^g\bar z)+(f_\alpha\circ\sigma_\alpha)(\tau_k^g,\bar z)-(f_\alpha\circ\sigma_\alpha)(t,\bar z)\to\varphi(g).
\end{equation}
By the relative triviality of $h(\tau,z)$ with respect to $(f_\alpha\circ\sigma_\alpha)(\tau,z)$, the sequence $\{h(\tau\tau_k^g,t\bar z)\}_{k\geq 1}$ converges for all $\tau,t\in T$.
Thus we can put
\begin{equation}\label{eq:def_g}
\bar h((\tau,g),t\bar z)=\lim_{k\to\infty} h(\tau\tau_k^g,t\bar z)=h(\tau,g(t\bar z))+\lim_{k\to\infty} h(\tau_k^g,t\bar z)
\end{equation}
for every $(\tau, g,t\bar z)\in T\times G\times Z$.
Suppose that there exist sequences $\{(\tau_k^{i},z_k^i)\}_{k\geq 1}\subset T\times Z$ for $i=1,2$ so that $z_k^i\to z$, $\tau_k^i z_k^i\to g(z)=\lim_{k\to\infty}g(z_k^i)$, and
\begin{equation*}
(f_\alpha\circ\sigma_\alpha)(\tau_k^{i},z_k^i)\to\varphi(g) \quad\textup{as}\enspace k\to\infty,
\end{equation*}
while for $i=1,2$ the limit points $\bar h_i=\lim_{k\to\infty} h(\tau_k^{i},z_k^i)\in\R_\infty$ are either distinct or both equal to $\infty$.
Then $(g(z),\varphi(g),\bar h_i)\in\mc D_{T,F} (z,0,0)$ for $i=1,2$, and for every $\tau'\in T$ follows from $g\in\textup{Aut}(Z,T)$ and the cocycle identity that
\begin{eqnarray*}
(g(\tau' z),\varphi(g)+(f_\alpha\circ\sigma_\alpha)(\tau',g(z))-(f_\alpha\circ\sigma_\alpha)(\tau',z),h(\tau',g(z))+\bar h_i-h(\tau',z))=\hspace{-5mm}\\
(g(\tau' z),\varphi(g),h(\tau',g(z))+\bar h_i-h(\tau',z))\in\mc D_{T,F}(\tau' z,0,0).
\end{eqnarray*}
Since $\bar{\mc O}_T(z)=Z$, either there are distinct points $a_1,a_2\in\R_\infty$ with $(g(\bar z),\varphi(g),a_i)\in\mc D_{T,F} (\bar z,0,0)$ or it holds that $(g(\bar z),\varphi(g),\infty)\in\mc D_{T,F} (\bar z,0,0)$.
In either case, since $\bar{\mc O}_{T,F} (\bar z,0,0)=\mc D_{T,F} (\bar z,0,0)$ in $Z\times(\R_\infty)^2$, this contradicts to the relative triviality of $h(\tau,z)$ with respect to $(f_\alpha\circ\sigma_\alpha)(\tau,z)$.
Therefore equality (\ref{eq:h_uni}) holds true, and the definition (\ref{eq:def_g}) extends uniquely from the $T$-orbit of $\bar z$ to a continuous mapping $\bar h:T\times G\times Z\longrightarrow\R$ since the action of $T\times G$ on $X$ and $\varphi$ are continuous.

For the cocycle identity let $(\tau_1,g_1),(\tau_2,g_2)\in T\times G$ be arbitrary with sequences $\{\tau_k^{g_1}\}_{k\geq 1},\{\tau_k^{g_2}\}_{k\geq 1}\subset T$.
By equality (\ref{eq:transf}) we select a sequence $\{k_l\}_{l\geq 1}\subset\N$ with
\begin{equation*}
\tau_{k_l}^{g_2}\tau_l^{g_1}\bar z=\tau_l^{g_1}\tau_{k_l}^{g_2}\bar z\to g_2(g_1(\bar z))\enspace\textup{and}\enspace (f_\alpha\circ\sigma_\alpha)(\tau_l^{g_1}\tau_{k_l}^{g_2},\bar z)\to\varphi(g_2)+\varphi(g_1)=\varphi(g_2 g_1)
\end{equation*}
as $l\to\infty$.
Thus we can put $\{\tau_l^{g_2 g_1}\}_{l\geq 1}=\{\tau_l^{g_1}\tau_{k_l}^{g_2}\}_{l\geq 1}$, and for every $t\in T$ the equality (\ref{eq:transf}) implies that $\tau_{k_l}^{g_2}\tau_2 t\bar z\to g_2(\tau_2 t\bar z)$, $\tau_l^{g_1}\tau_{k_l}^{g_2}\tau_2 t\bar z\to(g_2 g_1)(\tau_2 t\bar z)$, and
\begin{eqnarray*}
(f_\alpha\circ\sigma_\alpha)(\tau_1\tau_l^{g_1},\tau_{k_l}^{g_2}\tau_2 t\bar z)=\hspace{7.5cm}\\
(f_\alpha\circ\sigma_\alpha)(\tau_1,\tau_l^{g_1}\tau_{k_l}^{g_2}\tau_2 t\bar z)+(f_\alpha\circ\sigma_\alpha)(\tau_l^{g_1}\tau_{k_l}^{g_2},\tau_2 t\bar z)-(f_\alpha\circ\sigma_\alpha)(\tau_{k_l}^{g_2},\tau_2 t\bar z)\\
\to (f_\alpha\circ\sigma_\alpha)(\tau_1,g_2(\tau_2 t\bar z))+\varphi(g_2g_1)-\varphi(g_2) \quad\textup{as}\enspace l\to\infty.
\end{eqnarray*}
The uniqueness according to equality (\ref{eq:h_uni}) verifies that $\bar h((\tau_1,g_2 g_1 g_2^{-1}),g_2(\tau_2 t\bar z))=\lim_{l\to\infty}h(\tau_1\tau_l^{g_1},\tau_2\tau_{k_l}^{g_2}t\bar z)$, and therefore
\begin{eqnarray*}
\bar h((\tau_1,g_2 g_1 g_2^{-1}),g_2(\tau_2t\bar z))+\bar h((\tau_2,g_2),t\bar z)=\hspace{4.5cm}\\
=\lim_{l\to\infty}h(\tau_1\tau_l^{g_1},\tau_2\tau_{k_l}^{g_2}t\bar z)+\lim_{l\to\infty}h(\tau_2\tau_{k_l}^{g_2},t\bar z)=\\
=\lim_{l\to\infty}h(\tau_1\tau_l^{g_1}\tau_2\tau_{k_l}^{g_2},t\bar z)=\bar h((\tau_1\tau_2,g_2g_1),t\bar z) .
\end{eqnarray*}
We substitute $g_2^{-1} g_1 g_2$ for $g_1$ and obtain from $\bar{\mc O}_T(\bar z)=Z$ and the continuity of $\bar h$ that cocycle identity is valid.
\end{proof}

\begin{proof}[Proof of Proposition \ref{prop:res}]
Let $(Y_c,T)=\pi_c(X,T)$ be the flow defined by the connected components of the fibres of $\pi_Y$ (cf. \cite{MMWu}, Definition 2.2), and let $\rho$ be the homomorphism from $(Y_c,T)$ onto $(Y,T)=\rho(Y_c,T)$.
With a \emph{RIM} $\{\mu_{c,y}:y\in Y_c\}$ for the distal extension $(Y_c,T)=\pi_c(X,T)$ we define a cocycle $f_c(\tau,y)=\mu_{c,y}(f(\tau,\cdot))$ for every $(\tau,y)\in T\times Y_c$.
We fix a point $\bar x\in X$ with $\mc D_{T,f_\alpha\circ\pi_\alpha}(\tau\bar x,0)=\bar{\mc O}_{T,f_\alpha\circ\pi_\alpha}(\tau\bar x,0)$ for all $\tau\in T$.
By equality (\ref{eq:p_X}) and $\pi_c^{-1}(\pi_c(\tau\bar x))\subset\pi_Y^{-1}(\pi_Y(\tau\bar x))$ holds for all $\tau\in T$
\begin{equation*}
\bar{\mc O}_{T,f_\alpha\circ\pi_\alpha}(\bar x,0)\cap(\pi_c^{-1}(\pi_c(\tau\bar x))\times\R)=\pi_c^{-1}(\pi_c(\tau\bar x))\times\{(f_\alpha\circ\pi_\alpha)(\tau,\bar x)\} .
\end{equation*}
Let $F(\tau,x)$ be the $\R^2$-valued cocycle $(f_\alpha\circ\pi_\alpha,f)$.
We shall verify that
\begin{eqnarray}\label{eq:f_c}
\bar{\mc O}_{T,F} (\bar x,0,0)\cap(\pi_c^{-1}(\pi_c(\tau\bar x))\times\{(f_\alpha\circ\pi_\alpha)(\tau,\bar x)\}\times\R)=\nonumber\\
\{(x,(f_\alpha\circ\pi_\alpha)(\tau,\bar x),b_\tau(x)):x\in\pi_c^{-1}(\pi_c(\tau\bar x))\}
\end{eqnarray}
for every $\tau\in T$, in which $b_\tau:\pi_c^{-1}(\pi_c(\tau\bar x))\longrightarrow\R$ is a continuous function.
Indeed, for a sequence $\{\tau_k\}_{k\geq 1}\subset T$ with $\tau_k\tau\bar x\to x\in\pi_c^{-1}(\pi_c(\tau\bar x))$ and $(f_\alpha\circ\pi_\alpha)(\tau_k,\tau\bar x)\to 0$ follows by the relative triviality of $(f-f_\alpha\circ\pi_\alpha)(\tau,x)$ the existence and uniqueness of the limit $b_\tau(x)$ of $f(\tau_k,\tau\bar x)$.
It also follows that for every $\varepsilon>0$ there exists a $\delta>0$ so that for all $\tau\in T$ and $x,x'\in\pi_c^{-1}(\pi_c(\tau\bar x))$ with $d(x,x')<\delta$ holds $|b_\tau(x)-b_\tau(x')|<\varepsilon$.
Since the fibres of $\pi_c$ are connected, a covering of $X$ by $\delta$-neighbourhoods provides a constant $D>0$ with $|b_\tau(x)-b_\tau(x')|<D$ for all $\tau\in T$ and $x,x'\in\pi_c^{-1}(\pi_c(\tau\bar x))$.
Equality (\ref{eq:f_c}) shows for $x\in\pi_c^{-1}(\pi_c(\bar x))$ and $\tau\in T$ that $b_{\mathbf 1_T}(x)+f(\tau,x)=b_{\tau}(\tau x)$ and hence $|f(\tau,\bar x)-f(\tau,x)|\leq 2D$.
Since $(f-f_c\circ\pi_c)(\tau,x)$ assumes zero on each $\pi_c$-fibre, for all $\tau\in T$ holds $|(f-f_c\circ\pi_c)(\tau,\bar x)|< 2D$ so that this is a coboundary.

Due to Theorem 3.7 in \cite{MMWu} the extension from $(Y,T)$ to $(Y_c,T)$ is an isometric extension, and by Fact \ref{fact:reg_ex} there exists a compact group extension $(\tilde Y,T)$ of $(Y,T)=\rho(Y_c,T)$ by $G\subset\textup{Aut}(\tilde Y,T)$ so that $(Y_c,T)=\sigma(\tilde Y,T)$ is the orbit space of a compact subgroup $H\subset G$.
For every sequence $\{(\tau_k,\tilde y_k)\}_{k\geq 1}\subset T\times\tilde Y$ with $d_{\tilde Y} (\tilde y_k,\tau_k\tilde y_k)\to 0$ and $(f_\alpha\circ\rho_\alpha\circ\rho\circ\sigma)(\tau_k,\tilde y_k)\to 0$ holds $(f_c\circ\sigma)(\tau_k,\tilde y_k)\to 0$.
Otherwise, by Lemma \ref{lem:trans} there exists a sequence $\{(\tilde\tau_k,x_k)\}_{k\geq 1}\subset T\times X$ so that $d(x_k,\tilde\tau_k x_k)\to 0$, $(f_\alpha\circ\pi_\alpha)(\tilde\tau_k,x_k)\to 0$, and $(f_c\circ\pi_c)(\tilde\tau_k,x_k)\to\infty$, which contradicts to Proposition \ref{prop:max} and the boundedness of the transfer function between $f$ and $f_c\circ\pi_c$.
We can apply Lemma \ref{lem:tilde_coc} for the flow $(\tilde Y,T)$, the cocycle $h=f_c\circ\sigma$, the group $G\subset\textup{Aut}(\tilde Y,T)$, and the group homomorphism $\varphi\equiv 0$, and we obtain a real valued cocycle $\bar h((\tau,h),\tilde y)$ with $\bar h((\tau,\mathbf 1_G),\tilde y)=(f_c\circ\sigma)(\tau,\tilde y)$ for every $(\tau,\tilde y)\in T\times\tilde Y$.
We define a topological cocycle of $(Y,T)$ by $f_Y(\tau,y)=\mu_{Y,y}((f_c\circ\sigma)(\tau,\cdot))$, where $\{\mu_{Y,y}:y\in Y\}$ is the \emph{RIM} for the extension $(Y,T)=\rho\circ\sigma(\tilde Y,T)$.
From the cocycle identity $\bar h((\tau,g),\tilde y)=(f_c\circ\sigma)(\tau,g(\tilde y))+\bar h((\mathbf 1_T,g),\tilde y)=\bar h((\mathbf 1_T,g),\tau\tilde y)+(f_c\circ\sigma)(\tau,\tilde y)$ and the boundedness of $\bar h((\mathbf 1_T,g),\tilde y)$ for $(g,\tilde y)\in G\times\tilde Y$ we can conclude that $(f_c\circ\sigma-f_Y\circ\rho\circ\sigma)(\tau,\tilde y)$ is uniformly bounded and a coboundary.
Hence also $(f_c-f_Y\circ\rho)(\tau,y_c)$ is a coboundary, and the relative triviality of $f_Y(\tau,y)$ with respect to $(f_\alpha\circ\rho_\alpha)(\tau,y)$ can be verified as above for the cocycle $(f_c\circ\sigma)(\tau,\tilde y)$.
\end{proof}

\begin{proposition}\label{prop:inc}
The cocycle $(f_Y-f_\alpha\circ\rho_\alpha)(\tau,y)$ of the flow $(Y,T)$ can be extended to a cocycle $\bar f((\tau,t),y)$ of the $T\times\R$-flow $(Y,\{\psi^t\circ\tau:(\tau,t)\in T\times\R\})$ so that
\begin{equation*}
(f_Y-f_\alpha\circ\rho_\alpha)(\tau,y)=\bar f((\tau,0),y)
\end{equation*}
for every $(\tau,y)\in T\times Y$.
We put $L_\tau (x_\alpha,m)=(\tau x_\alpha,m)=\psi^{-f_\alpha(\tau,x_\alpha)}(\tau (x_\alpha,m))$ for $(x_\alpha,m)\in X_\alpha\times M=Y$.
Then for arbitrary $\bar x\in X_\alpha$ there exists a continuous function $\bar b:Y\longrightarrow\R$ with $\bar b(\rho_\alpha^{-1}(\bar x))=\{0\}$ so that for every $(\tau,y)\in T\times Y$ holds
\begin{equation}\label{eq:tf}
\bar f((\tau,-(f_\alpha\circ\rho_\alpha)(\tau,y)),y)=\bar b(\psi^{-(f_\alpha\circ\rho_\alpha)(\tau,y)} (\tau y))-\bar b(y) =\bar b(L_\tau y)-\bar b(y) ,
\end{equation}
and $(\tau,y)\mapsto\bar f((\tau,-(f_\alpha\circ\rho_\alpha)(\tau,y)),y)$ is a topological coboundary of the distal flow $(Y,\{L_\tau:\tau\in T\})$ with transfer function $\bar b$.
For every $(x_\alpha,m)\in Y$ and $t\in\R$ holds
\begin{equation}\label{eq:tF}
\bar f((\mathbf 1_T,t),(x_\alpha,m))= \bar f((\mathbf 1_T,t),(\bar x,m))+\bar b(x_\alpha,\phi^t(m))-\bar b(x_\alpha,m) .
\end{equation}
\end{proposition}

\begin{proof}
Since Lemma \ref{lem:tilde_coc} can be applied to the cocycle $h=f_Y-f_\alpha\circ\rho_\alpha$, the group $G=\{\psi^t:t\in\R\}\subset\textup{Aut}(Y,T)$, and the group homomorphism $\varphi=\textup{id}_\R$, it provides the cocycle $\bar f((\tau,t),y)$.
The mapping $f'((\tau,t),y)=\bar f((\tau,t-(f_\alpha\circ\rho_\alpha)(\tau,y)),y)$ fulfils
\begin{eqnarray*}
f'((\tau,t),\psi^{t'}(L_{\tau'} y))+f'((\tau',t'),y)=\hspace{55mm}\\
=\bar f((\tau,t-(f_\alpha\circ\rho_\alpha)(\tau,\psi^{t'}(L_{\tau'} y))),\psi^{t'}(L_{\tau'} y))+\bar f((\tau',t'-(f_\alpha\circ\rho_\alpha)(\tau',y)),y)=\hspace{-7mm}\\
=\bar f((\tau\tau',t+t'-(f_\alpha\circ\rho_\alpha)(\tau\tau',y)),y)=f'((\tau\tau',t+t'),y)
\end{eqnarray*}
and is thus a cocycle of the \emph{minimal} flow $(Y,\{\psi^t\circ L_\tau:(\tau,t)\in T\times\R\})$.
Now let $(\tau,y)\in T\times Y$ be arbitrary.
By equality (\ref{eq:o_flow}) and the density of $\widetilde{\tau}_{f_\alpha}$-transitive points in $X_\alpha$, there exists a sequence $\{(\tau_k,y_k)\}_{k\geq 1}\subset T\times Y$ so that $y_k\to\tau y$, $\tau_k y_k\to\psi^{-(f_\alpha\circ\rho_\alpha)(\tau,y)}(\tau y)$, and $(f_\alpha\circ\rho_\alpha)(\tau_k,y_k)\to -(f_\alpha\circ\rho_\alpha)(\tau,y)$ as $k\to\infty$, and by equality (\ref{eq:h_uni}) holds $(f_Y-f_\alpha\circ\rho_\alpha)(\tau_k,y_k)\to\bar f((\mathbf 1_T,-(f_\alpha\circ\rho_\alpha)(\tau,y)),\tau y)$.
Thus
\begin{equation*}
(f_Y-f_\alpha\circ\rho_\alpha)(\tau_k\tau,\tau^{-1} y_k)\to\bar f((\mathbf 1_T,-(f_\alpha\circ\rho_\alpha)(\tau,y)),\tau y)+(f_Y-f_\alpha\circ\rho_\alpha)(\tau,y)
\end{equation*}
as $k\to\infty$, and this limit coincides with $\bar f((\tau,-(f_\alpha\circ\rho_\alpha)(\tau,y)),y)=f'((\tau,0),y)$ due to the cocycle identity for $\bar f((\tau,t),y)$.
Since $f_Y(\tau,y)$ is relatively trivial with respect to $(f_\alpha\circ\rho_\alpha)(\tau,y)$, for every $\varepsilon>0$ there exists a $\delta>0$ so that for every $(\tau',y')\in T\times Y$ with $d_Y(y',\tau' y')<\delta$ and $|(f_\alpha\circ\rho_\alpha)(\tau',y')|<\delta$ holds $|(f_Y-f_\alpha\circ\rho_\alpha)(\tau',y')|<\varepsilon$.
From $\tau^{-1} y_k\to y$, $\tau_k y_k\to L_\tau y$, and $(f_\alpha\circ\rho_\alpha)(\tau_k\tau,\tau^{-1} y_k)\to 0$ follows for every $(\tau,y)\in T\times Y$ with $d_Y(y,L_\tau(y))<\delta$ that $f'((\tau,0),y)<\varepsilon$.
Fact \ref{fact:GH} implies that the cocycle $(\tau,y)\mapsto f'((\tau,0),y)$ of the distal flow $(Y,\{L_\tau:\tau\in T\})$ is a coboundary on the $\{L_\tau:\tau\in T\}$-orbit closure $X_\alpha\times\{m\}$ with transfer function $b_m:X_\alpha\longrightarrow\R$ for every $m\in M$.
Since $\delta>0$ is valid for all $\{L_\tau:\tau\in T\}$-orbit closures, the transfer functions $\{b_m:m\in M\}$ are uniformly equicontinuous.
We fix a point $\bar x\in X_\alpha$ and obtain from the cocycle identity for all $(\tau,t)\in T\times\R$ and $(x_\alpha,m)\in Y$ that
\begin{eqnarray*}
f_{\bar x}((\tau,t),(x_\alpha,m))=f'((\tau,t),(x_\alpha,m))-f'((\mathbf 1_T,t),(\bar x,m))=\\
=b_{\phi^t(m)}(\tau x_\alpha)-b_{\phi^t(m)}(\bar x)-b_{m}(\tau x_\alpha)+b_m(\bar x) .
\end{eqnarray*}
The function $f_{\bar x}((\tau,t),(x_\alpha,m))$ is also a cocycle of $(Y,\{\psi^t\circ L_\tau:(\tau,t)\in T\times\R\})$ and bounded on $T\times\R\times Y$, hence a coboundary with a transfer function $\bar b:Y\longrightarrow\R$ so that $\bar b(\rho_\alpha^{-1}(\bar x))=\{0\}$.
Now equality (\ref{eq:tf}) follows, and equality (\ref{eq:tF}) follows from $f'((\mathbf 1_T,t),(x_\alpha,m))=\bar f((\mathbf 1_T,t),(x_\alpha,m))$ for all $t\in\R$ and $(x_\alpha,m)\in Y$.
\end{proof}

With these prerequisites we can conclude the proof of our main result.

\begin{proof}[Proof of the structure theorem]
We let all  elements of the theorem and the flow $\{\psi^t:t\in\R\}\subset\textup{Aut}(Y,\{L_\tau:\tau\in T\})\cap\textup{Aut}(Y,T)$ be defined according to the Propositions \ref{prop:max}, \ref{prop:flow}, \ref{prop:res}, and \ref{prop:inc}.
We fix a point $\bar x\in X_\alpha$ so that $(\bar x,0)$ is transitive for $\widetilde\tau_{f_\alpha}$ and $\bar b(\rho_\alpha^{-1}(\bar x))=\{0\}$.
Then we define a cocycle $g(t,m)$ of the distal minimal flow $(M,\{\phi^t:t\in\R\})$ by
\begin{equation*}
g(t,m)=\bar f((\mathbf 1_T,t),(\bar x,m)) \quad\textup{for all}\enspace (t,m)\in\R\times M.
\end{equation*}
From equalities (\ref{eq:tf}) and (\ref{eq:tF}) follows for all $\tau\in T$ and $(x_\alpha,m)\in Y$ that
\begin{eqnarray*}
(f_Y-f_\alpha\circ\rho_\alpha)(\tau,(x_\alpha,m))=\bar f((\tau,0),(x_\alpha,m))=\hspace{3cm}\\
=\bar f((\mathbf 1_T,f_\alpha(\tau,x_\alpha)),L_\tau (x_\alpha,m))+\bar b(L_\tau (x_\alpha,m))-\bar b(x_\alpha,m)=\\
=g(f_\alpha(\tau,x_\alpha),m)+\bar b\circ \tau(x_\alpha,m)-\bar b(x_\alpha,m) .
\end{eqnarray*}
Hence equality (\ref{eq:f_Y2}) holds for the cocycle $f_Y(\tau,y)-\bar b(\tau y)+\bar b(y)$ cohomologous to $f_Y(\tau,y)$, and this cocycle will be substituted for $f_Y(\tau,y)$ henceforth.
For every sequence $\{(\tau_k,x_k)\}_{k\geq 1}\subset T\times X$ with $(f_\alpha\circ\pi_\alpha)(\tau_k,x_k)\to 0$ it holds also that $(f_Y\circ\pi_Y)(\tau_k,x_k)\to 0$ as $k\to\infty$, and thus identity (\ref{eq:p_X}) implies identity (\ref{eq:fpy}).

For every ordinal $\xi$ with $\alpha\leq\xi\leq\eta$ we can apply the Propositions \ref{prop:flow}, \ref{prop:res}, and \ref{prop:inc} to the distal minimal flow $(X_\xi,T)$ and the cocycle $f_\xi(\tau,x_\xi)$.
We obtain a factor $(Y_\xi,T)=\pi_{Y_\xi}(X_\xi,T)$ with $(X_\alpha,T)=\rho_\alpha^\xi(Y_\xi,T)$, an $\R$-flow $\{\psi_\xi^t:t\in\R\}\subset\textup{Aut}(Y_\xi,T)$, a cocycle $f_{Y_\xi}(\tau,y_\xi)$ of $(X_\xi,T)$, and a cocycle $\bar f_\xi((\tau,t),y_\xi)$ of the flow $(Y_\xi,T\times\R)$ extending the cocycle $(f_{Y_\xi}-f_\alpha\circ\rho^\xi_\alpha)(\tau,y_\xi)$.
Striving for a contradiction to the maximality of the ordinal $\alpha$ (cf. Proposition \ref{prop:max}), we assume that the cocycle $(\mathbbm 1+g)(t,m)$ of the minimal flow $(M,\{\phi^t:t\in\R\})$ is recurrent so that the cocycle $\bar f_\eta((\mathbf 1_T,t),y_\eta)+t$ of the minimal flow $((\rho^\eta_\alpha)^{-1}(\bar x),\{\psi_\eta^t:t\in\R\})$ is also recurrent.
We let $\beta$ be the minimal element of the non-empty set of ordinals
\begin{equation*}
\{\alpha\leq\xi\leq\eta: \bar f_\xi((\mathbf 1_T,t),y_\xi)+t\enspace\textup{is a recurrent cocycle of}\enspace ((\rho_\alpha^\xi)^{-1}(\bar x),\{\psi_\xi^t:t\in\R\})\} ,
\end{equation*}
with $\beta>\alpha$ since $f_{Y_\alpha}=f_\alpha$ and $\bar f_\alpha((\mathbf 1_T,t),y_\alpha)\equiv 0$.
We fix a point $\bar x_\beta\in(\pi_\alpha^\beta)^{-1}(\bar x)$ so that $(\pi_{Y_\beta}(\bar x_\beta),0)$ is a recurrent point for the skew product extension of the flow $((\rho_\alpha^\beta)^{-1}(\bar x),\{\psi_\beta^t:t\in\R\})$ by the cocycle $\bar f_\beta((\mathbf 1_T,t),y_\beta)+t$.
Then there exists a sequence $\{\bar\tau_k\}_{k\geq 1}\subset T$ with $f_\alpha(\bar\tau_k,\pi_\alpha^\beta(\bar x_\beta))\to\infty$ so that
\begin{equation*}
\bar f_\beta((\mathbf 1_T,f_\alpha(\bar\tau_k,\pi_\alpha^\beta(\bar x_\beta))),\pi_{Y_\beta}(\bar x_\beta))+f_\alpha(\bar\tau_k,\pi_\alpha^\beta(\bar x_\beta))=f_{Y_\beta}(\bar\tau_k,\pi_{Y_\beta}(\bar x_\beta))\to 0
\end{equation*}
for $k\to\infty$, and by the cohomology of the cocycles $f_\beta(\tau,x_\beta)$ and $(f_{Y_\beta}\circ\pi_{Y_\beta})(\tau,x_\beta)$ the sequence $f_\beta(\bar\tau_k,\bar x_\beta)$ is bounded.
Hence there exists a sequence $\{k_l\}_{l\geq 1}\subset\N$ so that $f_\alpha(\bar\tau_{k_{l+1}},\pi_\alpha^\beta(\bar\tau_{k_l}\bar x_\beta))\to\infty$, $d_\beta(\bar\tau_{k_{l+1}} \bar x_\beta,\bar\tau_{k_l}\bar x_\beta)\to 0$, and $f_\beta(\bar\tau_{k_l},\bar x_\beta)$ is convergent.
Then the sequence $\{(\tau_l,x_l)=(\bar\tau_{k_{l+1}}(\bar\tau_{k_l})^{-1},\bar\tau_{k_l}\bar x_\beta)\}_{l\geq 1}\subset T\times X_\beta$ fulfils $f_\alpha(\tau_l,\pi_\alpha^\beta(x_l))\to\infty$, $d_\beta(x_l,\tau_l x_l)\to 0$, and $f_\beta(\tau_l,x_l)\to 0$ for $l\to\infty$.
However, for every $\alpha\leq\xi<\beta$ holds
\begin{equation*}
\bar f_\xi((\mathbf 1_T,f_\alpha(\tau_l,\pi_\alpha^\beta(x_l))),\pi_{Y_\xi}\circ\pi_\xi^\beta(x_l))+f_\alpha(\tau_l,\pi_\alpha^\beta(x_l))=f_{Y_\xi}(\tau_l,\pi_{Y_\xi}\circ\pi_\xi^\beta(x_l))\to\infty
\end{equation*}
for $l\to\infty$.
Otherwise, since $|\bar f_\xi((\mathbf 1_T,t),(x_\alpha,m_\xi))-\bar f_\xi((\mathbf 1_T,t),(\bar x,m_\xi))|$ is uniformly bounded for all $t\in\R$, $x_\alpha\in X_\alpha$, and $m_\xi\in M_\xi$ (cf. identity (\ref{eq:tF})), there exists a non-trivial prolongation in the skew product of the minimal flow $((\rho_\alpha^\xi)^{-1}(\bar x),\{\psi_\xi^t:t\in\R\})$ and its cocycle $\bar f_\xi((\mathbf 1_T,t),y_\xi)+t$, which sufficient for its recurrence (cf. Lemma \ref{lem:tr_coc}).
Therefore also $f_\xi(\tau_l,\pi_\xi^\beta(x_l))\to\infty$ as $l\to\infty$, and depending on the type of the ordinal $\beta$ follows either from Lemma \ref{lem:c_t} or Lemma \ref{lem:lim} that $\widetilde\tau_{f_\beta}$ is point transitive, in contradiction to the maximality of $\alpha$.
\end{proof}

\begin{proof}[Proof of the structure of the topological Mackey action]
In the proof of the decomposition theorem it is verified that the topological Mackey actions for the cocycle $f(\tau,x)$ of $(X,T)$ and the transient cocycle $(\mathbbm 1+g)(t,m)$ of $(M,\{\phi^t:t\in\R\})$ are topologically isomorphic.
Let $\{(M_\xi,\{\phi_\xi^t:t\in\R\}):0\leq\xi\leq\theta\}$ be the normal \emph{I}-system for the distal minimal compact metric flow $(M,\{\phi^t:t\in\R\})$ with the homomorphisms $\sigma_\xi:M\longrightarrow M_\xi$.
For every ordinal $0\leq\xi\leq\theta$ a cocycle $g_\xi(t,m_\xi)$ of $(M_\xi,\{\phi_\xi^t:t\in\R\})$ is defined by a \emph {RIM}.
Let $\beta$ be the minimal element of the set
\begin{equation*}
\theta\in\{0\leq\xi\leq\theta:(g-g_\xi\circ\sigma_\xi)(t,m)\enspace\textup{is a coboundary of}\enspace(M,\{\phi^t:t\in\R\})\} .
\end{equation*}
The cocycle $(\mathbbm 1+g_\beta)(t,m_\beta)$ is transient, since the cocycle $(\mathbbm 1+g)(t,m)$ cohomologous to $(\mathbbm 1+g_\beta\circ\sigma_\beta)(t,m)$ is transient.
By Lemma \ref{lem:tr_coc} the right translation action $\{R_b:b\in\R\}$ acts minimally on the Fell compact space $D$ of orbits in $M_\beta\times\R$.
The mapping $\chi:M_\beta\longrightarrow D$ defined by $m_\beta\mapsto\mc O_{\phi_\beta,(\mathbbm 1+g_\beta)}(m_\beta,0)$ is Fell continuous, and for every $t\in\R$ holds $\chi\circ\phi_\beta^t(m_\beta)=R_{(\mathbbm 1+g_\beta)(t,m_\beta)}\circ\chi(m_\beta)$.
For $\beta=0$ the flow $(D,\{R_b:b\in\R\})$ is trivial and thus weakly mixing.
If $\beta\geq 1$, then $(D,\{R_b:b\in\R\})$ is a non-trivial minimal compact metric flow.
If it is not weakly mixing, then there exists a non-trivial equicontinuous factor $(D_1,\{\varphi^t:t\in\R\})=\nu(D,\{R_b:b\in\R\})$ with homomorphism $\nu$ (cf. \cite{KeRo}).
We shall use a generalised and relativised version of Theorem 1 in \cite{Eg} to obtain a contradiction to the minimality of $\beta$.
Since $(D_1,\{\varphi^t:t\in\R\})$ is a minimal and non-trivial flow, for each small enough $\varepsilon>0$ holds $\varphi^\varepsilon(d_1)\neq d_1$ for all $d_1\in D_1$.
We shall verify as a sub-lemma that there are no sequences $\{t_k\}_{k\geq 1}\subset\R$, $\{m_k\}_{k\geq 1},\{m'_k\}_{k\geq 1}\subset M_\beta$ so that $m_k\to\bar m$, $m'_k\to\bar m$, $\phi_\beta^{t_k}(m_k)\to\bar m'$, $\phi_\beta^{t_k}(m'_k)\to\bar m'$, and $g_\beta(t_k,m'_k)-g_\beta(t_k,m_k)\to\varepsilon$.
Indeed,
\begin{gather*}
\nu\circ\chi\circ\phi_\beta^{t_k}(m'_k)=\varphi^{(\mathbbm 1+g_\beta)(t_k,m'_k)}\circ\nu\circ\chi(m'_k)\to\varphi^{\varepsilon}(\lim \varphi^{(\mathbbm 1+g_\beta)(t_k,m_k)}\circ\nu\circ\chi(m'_k))
\end{gather*}
and $\lim\varphi^{(\mathbbm 1+g_\beta)(t_k,m_k)}\circ\nu\circ\chi(m'_k)=\lim\varphi^{(\mathbbm 1+g_\beta)(t_k,m_k)}\circ\nu\circ\chi(m_k)=\nu\circ\chi(\bar m')$, by the equicontinuity of $(D_1,\{\varphi^t:t\in\R\})$, imply that $\nu\circ\chi(\bar m')=\varphi^{\varepsilon}\circ\nu\circ\chi(\bar m')$, which contradicts to the choice of $\varepsilon$.

If $\beta=\gamma+1$ for some ordinal $0\leq\gamma<\theta$, then the sub-lemma implies the uniform equicontinuity of the mapping $m_\beta\mapsto g_\beta(t,m_\beta)$ restricted on the $\sigma_\gamma^\beta$-fibres and for all $t\in\R$.
Indeed, otherwise we can find sequences as above with $\sigma_\gamma^\beta(m_k)=\sigma_\gamma^\beta(m'_k)$ for all $k\geq 1$, and the condition on sufficiently small $\varepsilon$ can be fulfilled by the connectedness of the $\sigma_\gamma^\beta$-fibres.
Thus for all $t\in\R$ the cocycle $(g_\beta-g_\gamma\circ\sigma_\gamma^\beta)(t,m_\beta)$ is uniformly equicontinuous and assumes zero on every connected $\sigma_\gamma^\beta$-fibre.
By Fact \ref{fact:GH} this is then a coboundary, in contradiction to the minimality of $\beta$.

If $\beta$ is a limit ordinal, then the sub-lemma applies to sequences so that for every ordinal $0\leq\xi<\beta$ there exists an integer $k_\xi\geq 1$ with $\sigma_\xi^\beta(m_k)=\sigma_\xi^\beta(m'_k)$ for all $k\geq k_\xi$.
Hence there exists an ordinal $0\leq\zeta<\beta$ so that $|g_\beta(t,m_\beta)-g_\beta(t,m'_\beta)|<\varepsilon$ for all $m_\beta,m'_\beta\in M_\beta$ with $\sigma_\zeta^\beta(m_\beta)=\sigma_\zeta^\beta(m'_\beta)$ and for all $t\in\R$.
It follows that $(g_\beta-g_\zeta\circ\sigma_\zeta^\beta)(t,m_\beta)$ is a coboundary, in contradiction to the minimality of $\beta$.

The topological Mackey action of the transient cocycle $(\mathbbm 1+g)(t,m)$ is topologically isomorphic to the topological Mackey action of the cohomologous cocycle $(\mathbbm 1+g_\beta\circ\sigma_\beta)(t,m)$ (cf. the proof of  the decomposition theorem).
The weakly mixing flow $(D,\{R_b:b\in\R\})$ is a factor of the topological Mackey action of the cocycle $(\mathbbm 1+g_\beta\circ\sigma_\beta)(t,m)$, since for every $(m,s)\in M\times\R$ the mapping $\sigma_\beta\times\textup{id}_\R$ maps the orbit $\mc O_{\phi,(\mathbbm 1+g_\beta\circ\sigma_\beta)}(m,s)$ in $M\times\R$ to the orbit $\mc O_{\phi_\beta,(\mathbbm 1+g_\beta)}(\sigma_\beta(m),s)$ in $M_\beta\times\R$ continuously with respect to the Fell topologies.
Suppose that there exists two distinct orbits $\mc O, \mc O'$ in $M\times\R$ within the same $\sigma_\beta\times\textup{id}_\R$-fibre and $\{t_k\}_{k\geq 1}\subset\R$ is a sequence with $R_{t_k}\mc O\to \mc O''$ and $R_{t_k}\mc O'\to \mc O''$.
Since the mapping $t\mapsto (\mathbbm 1+g_\beta)(t,m_\beta)$ is onto $\R$ for every $m_\beta\in M_\beta$ (cf. Lemma \ref{lem:tr_coc}), there exists a point $\bar m\in M_\beta$ and distinct $m,m'\in\sigma_\beta^{-1}(\bar m)$ so that $(m,0)\in\mc O$ and $(m',0)\in\mc O'$.
Moreover, for every integer $k\geq 1$ we can select a real number $t'_k$ so that $(\mathbbm 1+g_\beta)(t'_k,\bar m)=t_k$, and therefore $(\phi^{t'_k}(m),0)\in R_{t_k}\mc O$ as well as $(\phi^{t'_k}(m'),0)\in R_{t_k}\mc O'$.
By changing to a subsequence we can suppose that $\phi^{t'_k}(m)\to m_1$ and $\phi^{t'_k}(m')\to m_2$ as $k\to\infty$ with $m_1\neq m_2$, by the distality of the flow $(M,\{\phi^t:t\in\R\})$.
However, since $(m_1,0),(m_2,0)\in\mc O''$, the point $(\bar m,0)$ is a periodic point in $(D,\{R_b:b\in\R\})$ in contradiction to the transience of the cocycle $(\mathbbm 1+g_\beta)(t,m_\beta)$.
We can conclude that $\sigma_\beta\times\textup{id}_\R$ is a distal homomorphism of the topological Mackey action of the cocycle $(\mathbbm 1+g_\beta\circ\sigma_\beta)(t,m)$ onto the weakly mixing flow $(D,\{R_b:b\in\R\})$.
\end{proof}

\textbf{Acknowledgement}: The author would like to thank Professor Jon Aaronson and Professor Eli Glasner for useful discussions and encouragement.


\begin{thebibliography}{99}
\bibitem[AkGl]{AkGl} E. Akin and E. Glasner, \emph{Topological ergodic decomposition and homogeneous flows}, Topological dynamics and applications, 43--52, Contemp. Math., 215, AMS, Providence, RI, 1998.

\bibitem[At]{A} G. Atkinson, \emph{A class of transitive cylinder transformations}, J. London Math. Soc. (2) \textbf{17} (1978), no. 2, 263--270.

\bibitem[Be]{Be}A.S. Besicovitch, \emph{A problem on topological transformations of the plane. II.}, Proc. Cambridge Philos. Soc. \textbf{47}, (1951).

\bibitem[Eg]{Eg} J. Egawa, \emph{Harmonizable minimal flows}, Funkcial. Ekvac. \textbf{25} (1982), no. 3, 357--362.

\bibitem[El]{El} R. Ellis, \emph{Lectures on Topological Dynamics}, W.A. Benjamin, New York, 1966.

\bibitem[Fe]{Fe} J. M. G. Fell, \emph{A Hausdorff topology for the closed subsets of a locally compact non-Hausdorff space}, Proc. Amer. Math. Soc. \textbf{13} (1962), 472-476.

\bibitem[Fu]{Fu} H. Furstenberg, \emph{The structure of distal flows}, Amer. J. Math. \textbf{85} (1963), 477-515.

\bibitem[Gl1]{Gl} E. Glasner, \emph{Relatively invariant measures}, Pacific J. Math. \textbf{58} (1975), no.2, 393--410.

\bibitem[Gl2]{Gl2} E. Glasner, \emph{Regular PI metric flows are equicontinuous}, Proc. Amer. Math. Soc. \textbf{114} (1992), No. 1, 269--277.

\bibitem[Gl3]{Gl3} E. Glasner, \emph{Topological ergodic decompositions and applications to products of powers of a minimal transformation}, J. Anal. Math. \textbf{64} (1994), 241--262.

\bibitem[GoHe]{G-H} W. H. Gottschalk and G. A. Hedlund, \emph{Topological dynamics}, American Mathematical Society Colloquium Publications, Vol. \textbf{36}, Providence, R. I., 1955.

\bibitem[Gr]{Gr} G. Greschonig, \emph{Nilpotent extensions of Furstenberg transformations}, Israel J. Math. \textbf{183} (2011), 381--397.

\bibitem[HLP]{HLP} L'. Hol\'a, S. Levi, and J. Pelant, \emph{Normality and paracompactness of the Fell topology}, Proc. Amer. Math. Soc. \textbf{127} (1999), no. 7, 2193--2197.

\bibitem[KeRo]{KeRo} H. B. Keynes and J. B. Robertson, \emph{Eigenvalue theorems in topological transformation groups}, Trans. Amer. Math. Soc. \textbf{139} (1969) 359--369. 

\bibitem[Ku]{K} K. Kuratowski, \emph{Topology. Vol. II.}, Academic Press, New York-London, 1968.

\bibitem[LemLes]{LL} M. Lema\'nczyk and E. Lesigne, \emph{Ergodicity of Rokhlin cocycles}, J. Anal. Math. \textbf{85} (2001), 43--86.

\bibitem[LemMe]{LM} M. Lema\'nczyk and M. Mentzen, \emph{Topological ergodicity of real cocycles over minimal rotations}, Monatsh. Math. \textbf{134} (2002), no. 3, 227--246.

\bibitem[LemPa]{LP} M. Lema\'nczyk and F. Parreau, \emph{Rokhlin extensions and lifting disjointness}, Ergodic Theory Dynam. Systems \textbf{23} (2003), no. 5, 1525--1550.

\bibitem[MaWu]{MMWu} D. McMahon and T. S. Wu, \emph{On the connectedness of homomorphisms in topological dynamics}, Trans. Amer. Math. Soc. \textbf{217} (1976), 257--270.

\bibitem[Sch]{Sch} K. Schmidt, \textit{Cocycles on ergodic transformation groups}, Macmillan Lectures in Mathematics, Vol. \textbf{1}, Macmillan Company of India, Ltd., Delhi, 1977.

\bibitem[dVr]{dVr} J. de Vries, \emph{Elements of topological dynamics}, Mathematics and its Applications, 257. Kluwer Academic Publishers Group, Dordrecht, 1993.
\end{thebibliography}
\end{document}